%% file: localizedwaves.tex
\documentclass[11pt,a4paper]{article}

\usepackage{fullpage}
\usepackage[T1]{fontenc} 
\usepackage{amsmath,amssymb}
\usepackage{amsopn,amsthm}
\usepackage{subcaption}

\usepackage{graphicx}
\usepackage{color}
\usepackage{psfrag}

\usepackage{tikz}
\usepackage{paralist}

\usepackage{mathrsfs} 
\usepackage{bm}
\usepackage{bbm}
\usepackage{cases}

\usepackage{algpseudocode} 
\usepackage{algorithm}
\floatstyle{plain} \newfloat{myalgo}{tbhp}{mya}

\usepackage[a]{esvect}
\usepackage{float}

\usepackage[font=small,labelfont=bf]{caption}

\usepackage{enumitem}

      {\end{algorithm}
    \end{minipage}
  \end{myalgo}}

\usepackage[colorlinks=true]{hyperref}
\hypersetup{urlcolor=blue, citecolor=blue, linkcolor=blue}

\allowdisplaybreaks

\newtheorem{theorem}{Theorem}[section]
\newtheorem{proposition}[theorem]{Proposition}
\newtheorem{lemma}[theorem]{Lemma}
\newtheorem{corollary}[theorem]{Corollary}

\theoremstyle{definition}

\theoremstyle{remark} \newtheorem{remark}[theorem]{Remark}

\numberwithin{equation}{section}
\numberwithin{figure}{section}
\numberwithin{algorithm}{section}
\numberwithin{table}{section}

\input{macros}

\begin{document}

\title{Localizing acoustic and electromagnetic waves\\ in space and time} 
\author{Roland Griesmaier$^*$ 
  and Soumen Senapati\footnote{Institut f\"ur
    Angewandte und Numerische Mathematik, Karlsruher Institut f\"ur
    Technologie, Englerstr.~2, 76131 Karlsruhe, Germany ({\tt
      roland.griesmaier@kit.edu, soumen.senapati@kit.edu}).}
}
\date{\today}

\maketitle

\begin{abstract}
  We study time-dependent acoustic and electromagnetic waves governed
  by the scalar wave equation or Maxwell's equations in a bounded
  three-dimensional domain. 
  We establish the existence of time-dependent boundary excitations that can
  be prescribed on any open subset of the boundary of the domain such
  that the associated waves are strongly localized in space in the
  sense that they possess arbitrarily large norms in a given subdomain
  and on a given time-interval, while remaining arbitrarily small in
  any other given subdomain for all times. 
  Similarly, we also show the existence of boundary data such that the
  associated waves are strongly localized in time in the sense that 
  they possess arbitrarily large norms in a given subdomain and on
  a given time-interval, while remaining arbitrarily small on the
  same subdomain but on any other prescribed time-interval. 
  In case that we have access to the possibly inhomogeneous
  coefficients in the wave equation or in the Maxwell system, we also
  give explicit constructions to obtain boundary data that generate
  these localized waves, and we comment on possible applications. 
\end{abstract}

{\small\noindent
  Mathematics subject classifications (MSC2010):
  35Q93  
  (35Q61 
  35L05) 
  \\\noindent 
  Keywords: Localized waves, focusing of waves, wave equation, Maxwell equations
  \\\noindent
  Short title: Localizing acoustic and electromagnetic waves
}

\section{Introduction.}
\label{sec:Introduction}
Localized or singular solutions have proven to be a fundamental tool 
in the study of uniqueness results and reconstruction algorithms for 
inverse boundary value problems for partial differential equations. 
Already very early uniqueness proofs for the inverse conductivity
problem in~\cite{KohVog1984,KohVog1985} rely on probing the unknown
conductivity distribution using sequences of boundary potentials that 
generate highly focused voltage distributions inside the domain.
One particular construction of such localized potentials for 
elliptic equations, which has been introduced in~\cite{Geb2008}, 
combines a duality argument for certain data-to-solution
operators with a unique continuation principle for solutions of the
underlying partial differential equation to establish the existence of
sequences of boundary currents that give rise to potentials inside the
domain possessing arbitrarily large norm on a prescribed subdomain, while
nearly vanishing on another given subdomain. 
These localized potentials have been successfully applied to establish
novel uniqueness results for inverse boundary value problems from
local boundary data in~\cite{Har2009,HarUlr2017}. 
The combination of such localized potentials with monotonicity 
relations for the associated Neumann-to-Dirichlet operators
has also been used to establish the theoretical foundation
of novel qualitative reconstruction methods for the inverse
conductivity problem in~\cite{HarUlr2013,Tam06,TamRub02}. 

Meanwhile, these localized potentials and the monotonicity based
uniqueness proofs and reconstruction methods for the inverse
conductivity problem have been extended to time-harmonic wave 
equations describing acoustic, electromagnetic and elastic waves at
fixed frequency 
in~\cite{EbePoh2024,HarLinLiu2018,HarPohSal2019b,HarPohSal2019a}.
It has also been noted that the construction of localized solutions 
in the fixed frequency case is connected to the study of quantitative 
Runge estimates as considered in~\cite{Poh2022,RueSal2019}. 
Corresponding localized waves and monotonicity based
reconstruction methods for time-harmonic inverse scattering
problems on unbounded domains have been analyzed
in~\cite{AlbGri2023,GriHar2018,HarXia2026}.

The aim of this work is to extend the concept of localized 
potentials from \cite{Har2009,HarLinLiu2018,HarPohSal2019a} to 
initial boundary value problems for the time-dependent wave 
equation and for time-dependent Maxwell's equations.
Since time offers an additional degree of freedom, we consider the
following two cases:

(i) \emph{Focusing in space:}
We show that there exist time-dependent excitations on some part of
the boundary of the domain that generate waves in the interior of the
domain possessing arbitrarily large norms in a prescribed
subdomain and on a given time-interval, while nearly vanishing in 
another prescribed subdomain for all times. 

(ii) \emph{Focusing in time:}
We establish the existence of time-dependent boundary excitations on 
some part of the boundary of the domain that generate waves in the 
interior of the domain possessing arbitrarily large norms in a prescribed 
subdomain and on a given time-interval, while nearly vanishing on the
same subdomain but on another prescribed time-interval.

Our proofs combine duality arguments for data-to-solution operators 
with unique continuation principles for solutions to the time-dependent 
wave equation and Maxwell's equations. 
In particular, we invoke Tataru's seminal result \cite{Tat95,Tat99} 
on global unique continuation for the wave equation in optimal time, 
which generalized earlier important contributions from~\cite{Hor92} 
and~\cite{Rob91}. 
For an introduction to such unique continuation results and 
related applications, we ask the reader to consult~\cite{KatKurLas2001}
and~\cite{LauLea2023}. 
For some principally diagonalizable systems,  the unique continuation 
result in \cite{Tat95} was extended in \cite{EINT2002,Ell2003}. 
This is what we use in our results for Maxwell's equations. 
However, it requires stronger smoothness assumptions on the electric 
permittivity and the magnetic permeability than those needed for the 
rest of the argument.
Moreover, the finite speed of propagation of solutions to the wave 
equation and to Maxwell's equations leads to natural restrictions on 
the time-intervals for rising and perhaps disappearing of the 
localized waves in terms of the distances of the respective 
subdomains from the part of the boundary, where the boundary values 
are excited. 
As usual, these distances are to be measured with respect to the 
travel time metric. 

In case that we have access to the possibly inhomogeneous parameters
in the wave equation or in the time-dependent Maxwell system, we also 
provide explicit constructions of the boundary excitations that can
be used to generate localized waves.
These constructions can immediately be translated into numerical 
algorithms. 
In this context, we also note that iterated time reversal has been 
applied in~\cite{BKLS2008,DahKirLas2009,KKLO2021} to develop algorithms 
that use observations of the hyperbolic Dirichlet-to-Neumann map for 
the wave equation to focus scalar waves to a delta distribution at 
some fixed times even in unknown media. 
However, then the focusing is done in travel time coordinates. 
An advantage of our results might be the possibility to prescribe
regions away from the focus area where the amplitude of the waves are
kept at arbitrary small levels for all times. 

A natural extension of the cases (i) and (ii) described 
above would be the construction of boundary excitations on some 
part of the boundary that generate waves in the interior of the 
domain possessing arbitrarily large norms in a prescribed time-varying
subdomain and on some given time-interval, while nearly vanishing in 
another prescribed time-varying subdomain for all times.  
However, we are unable to address this question at the moment.  
The main difficulty here is the unavailability of a suitable global 
version of a Tataru-type unique continuation result for the 
time-dependent wave equation or Maxwell's equations from local Cauchy 
data for open sets with time-varying boundary.

Besides possible applications of localized solutions for 
time-dependent wave or Maxwell's equations in uniqueness proofs 
for inverse boundary value problems or in monotonicity-based 
qualitative reconstruction methods, which have been our main 
motivation for this work, several other potential applications
have for instance been proposed 
in~\cite{DahKirLas2009,HarLinLiu2018,KKLO2021}.
These include ultrasound induced heating, inductive charging, or 
secure communication. More specifically, localized
acoustic or elctromagnetic waves may be useful in telecommunication to 
securely deliver a message to a prescribed location at some designated time
while avoiding the undesired area throughout the transmission process. 
They may be further utilized to propose 
an efficient energy transfer mechanism based on electromagnetic induction
by concentrating electromagnetic fields near a specific target region.
In addition, the use of localized or enhanced wave-fields may 
be also advantageous in many therapeutic and biomedical applications, where some diseased 
area or unwanted anomalies are to be removed or destroyed by employing focused 
ultrasound waves without affecting adjacent healthy tissues and organs.

The remainder of this article proceeds as follows.
In Section~\ref{sec:Setting} we introduce our notation and the two 
initial boundary value problems (IBVP) for the wave equation and for 
Maxwell's equations that we use as basic models for acoustic and 
electromagnetic wave propagation throughout this work. 
We also briefly discuss the well-posedness of these problems.  
Then, in Section~\ref{sec:WaveEquation}, we establish the existence of
localized solutions for the IBVP for the wave equation and comment
on their construction. 
Localized solutions for the IBVP for Maxwell's equations are 
developed in Section~\ref{sec:MaxwellEquations}. 
In the Appendix we collect some abstract functional analytic results
that are used in Sections~\ref{sec:WaveEquation}
and~\ref{sec:MaxwellEquations}.

\section{The mathematical setting.}
\label{sec:Setting}
We start by introducing some notation concerning the two IBVP that we 
study in the rest of this work. 
Let $T>0$ and let $\Omega\subset\Rd$ be a bounded domain
with~$C^2$-smooth boundary~$\di\Omega$. 
Furthermore, let $\Gamma\subseteq\di\Omega$ be a relatively open
subset that will represent the support of all boundary excitations in the
following. 

The first problem that we discuss is an acoustic wave equation. 
Denoting by $c \in C^1(\overline\Omega)$ and~$q\in L^\infty(\Omega)$ the 
wave speed and an external potential, respectively, we 
consider the equation 
\begin{subequations}
  \label{eq:IBVP_acoustic}
  \begin{equation}
    c^{-2}(x) \di_t^2 u_f - \Delta u_f  + q(x)u_f \,=\, 0 
    \qquad \text{ in } \Omega\times (0,T) \,,
  \end{equation}
  together with the homogeneous initial conditions
  \begin{equation}
    u_f|_{t=0} \,=\, \di_t u_f|_{t=0} \,=\, 0 
    \qquad \text{in } \Omega \,,
  \end{equation}
  and the Dirichlet boundary condition
  \begin{equation}
    u_f|_{\di\Omega\times(0,T)}
    \,=\, f \text{ on } \Gamma\times (0,T)
    \quad\text{and}\quad
    u_f|_{\di\Omega\times(0,T)}
    \,=\, 0 \text{ on } (\di\Omega\setminus\Gamma)\times (0,T)
  \end{equation}
\end{subequations}
for some $f\in L^2(\Gamma\times(0,T))$.

Throughout we assume that the wave speed is bounded away from zero;
that is $c \geq c_-$ in $\Omega$ for some positive
constant~$c_->0$. 
Following~\cite[p.~72]{KatKurLas2001}, we call
$u_f\in L^2(\Omega\times(0,T))$ a weak solution of the IBVP
\eqref{eq:IBVP_acoustic} if, for any $\psi\in H^2(\Omega\times(0,T))$
with $\psi|_{\red{\partial\Omega}\times(0,T)}=0$ and $\psi|_{t=T}=\di_t \psi|_{t=T}=0$,
\begin{equation*}
  \int_0^T \int_\Omega u_f
  \bigl(  c^{-2}(x)\di_t^2 \psi - \Delta \psi + q(x)\psi\bigr) \dx \dt
  \,=\, \int_0^T \int_\Gamma f \di_\bfnu \psi \ds_x \dt \,. 
\end{equation*}
As usual $\bfnu$ denotes the unit outward normal on $\di\Omega$ and 
$\di_\bfnu$ the associated normal derivative.

\begin{proposition}
  \label{pro:WellPosedness_IBVP_acoustic}
  For any $f\in L^2(\Gamma\times(0,T))$, there exists a
  unique weak solution $u_f$ of~\eqref{eq:IBVP_acoustic}. 
  This solution satisfies~$u_f \in C ([0,T]; L^2(\Omega) )$, and it
  depends continuously on the data.
\end{proposition}

\begin{proof}
  Under the aforementioned hypotheses this result can be inferred
  from Theorem~2.30 and Lemma~2.42 in~\cite{KatKurLas2001} (see
  also~\cite{LasLioTri1986}). 
\end{proof}

The second problem that we consider is a time-dependent Maxwell
system.
Throughout, we denote $\bfL^2(\Omega) := L^2(\Omega;\R^3)$ and 
\begin{equation*}
  \bfH(\curl;\Omega)
  \,:=\, \left\{\bfu\in \bfL^2(\Omega)
    \;|\; \curl\bfu \in \bfL^2(\Omega) \right\} \,. 
\end{equation*}
with the norms defined accordingly. 
We also recall the tangential trace operators
\begin{align*}
  \gamma_\tau &:\, \bfH(\curl;\Omega)
                \to \HminhalfdivOmega \,, \quad
                \bfu \mapsto \bfnu\times\bfu |_{\di\Omega} \,, \\
  \pi_\tau &:\, \bfH(\curl;\Omega)
             \to \HminhalfcurlOmega \,, \quad
             \bfu \mapsto \bfnu \times (\bfu\times\bfnu) |_{\di\Omega} \,.
\end{align*}
For a detailed discussion on this, and in particular on the trace
spaces $\HminhalfdivOmega$ and $\HminhalfcurlOmega$, we refer
to~\cite[Sec.~5.1]{KirHet2015}. 
With $\bfH_0(\curl;\Omega)$, we denote those $\bfu\in \bfH(\curl;\Omega)$ which satisfy $\nu\times\bfu\vert_{\partial\Omega}=0$ in $\HminhalfdivOmega$.

Denoting by~$\eps,\mu\in L^\infty(\Omega)$ the
electric permittivity and the magnetic permeability,
respectively, we consider the system of equations
\begin{subequations}
  \label{eq:IBVP_Maxwell}
  \begin{align}
    \eps(x)\, \di_t \bfE_\bff - \curl \bfH_\bff
    \,=\, 0 \,, \quad 
    \mu(x)\, \di_t \bfH_\bff + \curl \bfE_\bff
    \,=\, 0 \qquad \text{in } \Omega \times(0,T) \,, 
  \end{align}
  together with the homogeneous initial conditions
  \begin{equation}
    \bfE_\bff|_{t=0}
    \,=\, \bfH_\bff|_{t=0}
    \,=\, 0 
    \qquad \text{in } \Omega \,,
  \end{equation}
  and the boundary condition
  \begin{equation}
    \bfnu\times\bfE_\bff|_{\di\Omega\times(0,T)}
    \,=\, \bff \text{ on } \Gamma\times (0,T)
    \quad\text{and}\quad
    \bfnu\times\bfE_\bff|_{\di\Omega\times(0,T)}
    \,=\, 0 \text{ on } (\di\Omega\setminus\Gamma)\times (0,T)
  \end{equation} 
\end{subequations}
for some $\bff\in \Hcal^1_0([0,T];\HminhalfdivGamma)$.
Here~$\HminhalfdivGamma$ denotes the closure of tangential fields in
$C^2_0(\Gamma;\Rd)$ with respect to $\|\cdot\|_{\HminhalfdivOmega}$. 
Furthermore, for any Hilbert space $X$ we use the notation
\begin{equation*}
  \Hcal^1_0([0,T];X)
  \,:=\, \bigl\{ \bff\in H^1((0,T);X) \;\big|\; \bff(0)=0 \bigr\} \,.
\end{equation*}
In view of the vanishing initial condition, we consider the following
inner product on $\Hcal^1_0([0,T];X)$, 
\begin{equation*}
  \langle \bff, \bfg \rangle_{\Hcal^1_0([0,T];X)}
  \,:=\, \int_{0}^T \bigl\langle \di_t \bff(t) ,
  \di_t\bfg(t) \bigr\rangle_{X} \dt \qquad
  \text{for all } \bff, \bfg \in \Hcal^1_0([0,T];X) \,,
\end{equation*}
where $\langle \cdot , \cdot \rangle_{X}$ denotes the inner product of~$X$.
For the rest of this article, we suppress explicit reference to
the underlying function space in the notation
$\langle \cdot , \cdot \rangle$; it will be understood from the context. 

Throughout we assume that the electric permittivity and the 
magnetic permeability are bounded away from zero;  that is
$\eps\geq\eps_-$ and $\mu\geq\mu_-$ a.e.\ on $\Omega$ for some positive
constants~$\eps_-,\mu_->0$. 
Following~\cite{Ant2025}, we call $(\bfE_\bff,\bfH_\bff)$
with\footnote{For a given Banach space $X$, we denote its topological
  dual by $X^*$, and accordingly $\langle\cdot,\cdot\rangle_{X}$
  also denotes the duality pairing between $X^*$ and $X$.}
\begin{align*}
  \bfE_\bff &\in \Hcal^1_0([0,T];\bfH_0(\curl;\Omega)^*)
              \cap L^\infty((0,T);\bfL^2(\Omega)) \,, \\
  \bfH_\bff &\in \Hcal^1_0([0,T];\bfH(\curl;\Omega)^*)
              \cap L^\infty([0,T];\bfL^2(\Omega)) \,,
\end{align*}
a weak solution to the IBVP \eqref{eq:IBVP_Maxwell} if,
for all $t\in (0,T)$, it satisfies
\begin{subequations}
  \label{eq:WeakFormulation_IBVP_Maxwell}
  \begin{align}
    \bigl\langle \eps(x) \di_t\bfE_\bff(t), \bfPhi \bigr\rangle
    - \int_\Omega \bfH_\bff(t) \cdot\curl \bfPhi \dx
    &\,=\, 0 \,,
    &&\text{for all } \bfPhi \in  \bfH_0(\curl;\Omega) \,, \\
    \bigl\langle \mu(x) \di_t\bfH_\bff(t) , \bfPsi \bigr\rangle
    + \int_\Omega \bfE_\bff(t) \cdot \curl \bfPsi \dx
    &\,=\, - \left\langle \bff(t), \pi_\tau[\bfPsi]\right\rangle
    &&\text{for all } \bfPsi \in \bfH(\curl;\Omega) \,.
  \end{align}
\end{subequations}
Let us underline that $\langle\cdot,\cdot\rangle$ on the right hand
side of the second equation in \eqref{eq:WeakFormulation_IBVP_Maxwell}
denotes the dual pairing between $\HminhalfdivOmega$ and
$\HminhalfcurlOmega$. 
However, the same notation appearing on the left hand side of the
first and second equation of \eqref{eq:WeakFormulation_IBVP_Maxwell}
denotes the dual pairing
between~$\bfH_0(\curl;\Omega)^*$ and $\bfH_0(\curl;\Omega)$ as well
as~$\bfH(\curl;\Omega)^*$ and $\bfH(\curl;\Omega)$, respectively.  

\begin{proposition}
  \label{pro:WellPosedness_IBVP_Maxwell}
  For any $\bff\in \Hcal^1_0([0,T]; \HminhalfdivGamma)$, there exists a
  unique weak solution~$(\bfE_\bff,\bfH_\bff)$
  of~\eqref{eq:IBVP_Maxwell}. 
  This solution satisfies
  $(\bfE_\bff,\bfH_\bff)\in C([0,T]; \bfL^2(\Omega)^2)$, and it depends
  continuously on the data.
  Furthermore, we have
  $\div(\eps(x)\bfE_\bff(t))=\div(\mu(x)\bfH_\bff(t))=0$ for all~$t\in [0,T]$.
\end{proposition}

\begin{proof}
  Thanks to the surjectivity of the tangential trace
  $\gamma_\tau: \HcurlOmega \to \HminhalfdivOmega$,
  we can consider a lifting operator
  $G_\tau : \Hcal^1_0([0,T]; \HminhalfdivGamma)
  \mapsto \Hcal^1_0([0,T]; \bfH(\curl;\Omega))$ satisfying
  \begin{equation*}
    \gamma_\tau[G_\tau[\bff](t)]
    \,=\, \bff(t) \qquad \text{for all } t \in (0,T) \,.
  \end{equation*}
  In view of this, we define
  $\widetilde{\bfE}_\bff = \bfE_\bff - G_\tau[\bff]$ and
  recast~\eqref{eq:WeakFormulation_IBVP_Maxwell}, for
  all~$t\in (0,T)$, to the following weak formulation
  \begin{subequations}
    \label{eq:transformed_IBVP_Maxwell}
    \begin{align}
      \bigl\langle \eps(x) \di_t\widetilde{\bfE}_\bff(t) ,
      \bfPhi \bigr\rangle
      -  \int_\Omega \bfH_\bff(t) \cdot \curl\bfPhi \dx
      &\,=\, - \int_\Omega \eps(x) \di_tG_\tau[\bff]\cdot\bfPhi \dx \,, \\
      \bigl\langle \mu(x) \di_t\bfH_\bff(t) ,
      \bfPsi \bigr\rangle
      + \int_\Omega \widetilde{\bfE}_\bff(t) \cdot\curl \bfPsi \dx
      &\,=\, - \int_\Omega \curl G_\tau[\bff]\cdot\bfPsi \dx \,,
    \end{align}
  \end{subequations}
  for all~$ \bfPhi \in  \bfH_0(\curl;\Omega)$
  and~$\bfPsi \in \bfH(\curl;\Omega)$. 
  Still imposing homogeneous initial
  conditions~$\widetilde{\bfE}_\bff|_{t=0} =\bfH_\bff|_{t=0}=0$,
  the existence and uniqueness of a
  solution~$(\widetilde{\bfE}_\bff,\bfH_\bff)$ with 
  \begin{align*}
    \widetilde{\bfE}_\bff
    &\in \Hcal^1_0([0,T];\bfH_0(\curl;\Omega)^*)
      \cap L^\infty((0,T);\bfL^2(\Omega)) \,, \\
    \bfH_\bff
    &\in \Hcal^1_0([0,T];\bfH(\curl;\Omega)^*)
      \cap L^\infty((0,T);\bfL^2(\Omega)) \,,
  \end{align*}
  to~\eqref{eq:transformed_IBVP_Maxwell} can be
  established using the Galerkin method, as done 
  in~\cite[Thm.~2]{Ant2025},
  or using semigroup theory (see~\cite[Thm.~5.3]{KirRie2016} for such a
  result under slightly stronger regularity assumptions for the source
  term). 
  This in turn implies that the IBVP~\eqref{eq:IBVP_Maxwell}
  admits a unique weak
  solution~$(\bfE_\bff,\bfH_\bff)\in C([0,T];\bfL^2(\Omega)^2)$ for
  any~$\bff\in  \Hcal^1_0([0,T]; \HminhalfdivGamma)$
  (see \cite[Cor.~1]{Ant2025}).
  The continuous dependence of this solution on the data can be
  inferred from \cite[Thm.~2]{Ant2025}. 
  From the weak formulation
  \eqref{eq:WeakFormulation_IBVP_Maxwell}, it also follows that
  \begin{equation*}
    \div \bigl( \eps(x)\bfE_\bff(t) \bigr)
    \,=\, \div \bigl( \mu(x)\bfH_\bff(t) \bigr)
    \,=\, 0 \qquad
    \text{for all } t\in[0,T] \,.
  \end{equation*}
  For a proof of the latter, we can rely on an argument similar
  to~\cite[Pro.~1]{Ant2025}. 
\end{proof}

\section{Localized solutions for the acoustic wave equation.}
\label{sec:WaveEquation}
We discuss the existence and construction of localized solutions to the
IBVP for the wave equation~\eqref{eq:IBVP_acoustic}. 
To simplify the presentation, we denote for any subset $D\subs\Omega$ or 
$\Sigma \subs \di\Omega$ and for~$0 \leq a < b\leq T$, 
\begin{equation*}
  D_{a,b} \,:=\, D \times (a,b) \,, \;
  \Sigma_{a,b} \,:=\, \Sigma \times (a,b)
  \quad\text{and}\quad
  D_b \,:=\, D \times (0,b) \,, \;
  \Sigma_b \,:\,= \Sigma \times (0,b) \,.
\end{equation*}
Let us also introduce some terminologies regarding the
travel time metric associated to the wave equation \eqref{eq:IBVP_acoustic} 
and the Maxwell system \eqref{eq:IBVP_Maxwell}. 
In the latter case the wave speed is given by~${c = 1/\sqrt{\eps\mu}}$. 
We denote $\dist_{\Omega}(x,y)$ as the Riemannian distance
\begin{equation*}
  \dist_{\Omega}(x,y)
  \,:=\, \inf_{\gamma} \int_\alpha^\beta
  \frac{|\gamma'(t)|}{c(\gamma(t))} \dt \,, 
\end{equation*}
between two points $x,y \in \ol\Omega$, where the infimum is
taken over all smooth curves $\gamma$ in $\Omega$
satisfying~$\gamma(\alpha)=x$ and $\gamma(\beta)=y$. 

Moreover, let
\begin{equation*}
  \dist_{\Omega}(x,\Gamma)
  \,:=\, \inf_{y\in\Gamma}\dist_{\Omega}(x,y) \,,
  \quad x\in\Omega \,,
  \qquad \text{and} \qquad
  \dist_{\Omega}(\Omega, \Gamma)
  \,:=\, \sup_{x\in\ol\Omega} \dist_{\Omega}(x,\Gamma) \,,
\end{equation*}
and $\diam(\Omega) := \sup_{x,y\in\ol\Omega}\dist_{\Omega}(x,y)$. 
For any open set $\mathcal{O}\Subset\Omega$ having Lipschitz boundary with~$\Omega_{\mathcal{O}}:=\Omega\setminus\overline{\mathcal{O}}$ being connected, we define similar quantities like $\dist_{\Omega_{\mathcal{O}}}(x,y),\, \dist_{\Omega_{\mathcal{O}}}(x,\Gamma)$, and~$\dist_{\Omega_\mathcal{O}}(\Omega_{\mathcal{O}},\Gamma)$ as above.\footnote{Here, the Lipschitz regularity guarantees that $\dist_{\Omega_\Ocal}(\Omega_\Ocal, \Gamma) < \infty$.}

We define 
\begin{equation}
  \label{eq:time-reversal_and_translation}
  \Rcal_\tau[g](\ph,t) \,:=\, g(\ph,\tau-t) 
  \qquad \text{and} \qquad
  \mathcal{T}_\tau[g](\ph,t) \,:=\, g(\ph,t-\tau) \,,  
\end{equation}
which, with respect to the time-level $t=\tau\in\R$, represent the 
time-reversal and time-translation operator, respectively.

Our analysis of localized waves for the acoustic wave
equation relies on a careful investigation of the range spaces of the
adjoints of certain restricted solution operators for
the~IBVP~\eqref{eq:IBVP_acoustic}. 
For any bounded open subset $B\subset\Omega$ and
for $0\leq a< b\leq T$ we introduce
\begin{equation}
  \label{eq:DefLab}
  \L_{B_{a,b}}:\, L^2(\Gamma_T) \to L^2(B_{a,b}) \,,
  \quad  f \mapsto  u_f|_{B_{a,b}} \,, 
\end{equation}
where $u_f$ denotes the weak solution of~\eqref{eq:IBVP_acoustic}.
In the following lemma we identify the adjoint of this operator. 

\begin{lemma}
  \label{lmm:AdjointLab}
  The adjoint of the operator $\L_{B_{a,b}}$ from \eqref{eq:DefLab} is
  given by
  \begin{equation}
    \label{eq:AdjointLab}
    \L^*_{B_{a,b}} :\, L^2(B_{a,b}) \to  L^2(\Gamma_T) \,,
    \quad g \mapsto \di_{\bfnu} v_g|_{\Gamma_T} \,,
  \end{equation}
  where $v_g \in C([0,T];H^1_0(\Omega))\cap C^1([0,T];L^2(\Omega))$ is
  the unique weak solution of the~IBVP 
  \begin{subequations}
    \label{eq:adjointIBVP_acoustic}
    \begin{align}
      c^{-2}(x) \di_t^2 v_g - \Delta v_g + q(x)v_g
      &\,=\, g 
      &&\text{in } {\Omega_T} \,, \\
      v_g|_{t=T} \,=\, \di_t v_g|_{t=T}
      &\,=\, 0 
      &&\text{in } \Omega \,, \\
      v_g|_{(\di\Omega)_T}
      &\,=\, 0
      &&\text{on } (\di\Omega)_T \,. 
    \end{align}
  \end{subequations}
  Since this solution satisfies
  $\di_\bfnu v_g|_{(\di\Omega)_T} \in L^2((\di\Omega)_T)$, the
  operator $\L^*_{B_{a,b}}$ is well-defined. 
\end{lemma}

\begin{proof}
  Taking time-reversal into consideration, i.e., considering
  $\Rcal_T v_g$ instead of $v_g$ and $\Rcal_T g$ instead of $g$, we
  can apply~\cite[Thm.~2.30]{KatKurLas2001} to see that, for any
  $g\in L^2(B_{a,b})$ and after extending this function by zero to all
  of $\Omega\times(0,T)$, the~IBVP~\eqref{eq:adjointIBVP_acoustic}
  has a unique weak solution 
  ${v_g \in C([0,T];H^1_0(\Omega))\cap C^1([0,T];L^2(\Omega))}$ which
  satisfies the hidden regularity condition
  $\di_\bfnu v_g|_{(\di\Omega)_T} \in L^2((\di\Omega)_T)$.
  
  Now let $f\in L^2(\Gamma_T)$ and $g\in L^2(B_{a,b})$, and denote by
  $u_f$ and $v_g$ the corresponding weak solutions of
  \eqref{eq:IBVP_acoustic} and \eqref{eq:adjointIBVP_acoustic},
  respectively. 
  We denote by $\{f_k\}_{k\in\N}\subs C^\infty_0([0,T];C^2_0(\Gamma))$ and
  by $\{g_k\}_{k\in\N} \subs C^\infty_0([0,T];C^\infty_0(B_{a,b}))$ smooth
  approximations of $f$ and $g$, respectively.
  Then \cite[Thm.~2.45]{KatKurLas2001} shows that the associated
  solutions $u_{f_k}$ and $v_{g_k}$ of \eqref{eq:IBVP_acoustic} and
  \eqref{eq:adjointIBVP_acoustic} satisfy 
  $u_{f_k}\in C([0,T];H^2(\Omega))\cap C^2([0,T];L^2(\Omega))$ and
  $v_{g_k}\in C^\infty([0,T];C^\infty(\Omega))$, respectively.
  Furthermore, it follows from the continuous dependence
  of these solutions on the data (see~\cite[Thm.~2.30]{KatKurLas2001}) 
  that $u_f = \lim_{k\to\infty} u_{f_k}$ in $L^2(B_{a,b})$
  and~$\di_\bfnu v_g = \lim_{k\to\infty}\di_\bfnu v_{g_k}$
  in~$L^2((\di\Omega)_T)$.
  Using integration by parts, the strong formulations
  of~\eqref{eq:IBVP_acoustic} and \eqref{eq:adjointIBVP_acoustic}, and
  a passage to the limit, we observe 
  \begin{equation*}
    \begin{split}
      \bigl\langle \L_{B_{a,b}}[f] , g \bigr\rangle_{L^2(B_{a,b})}
      &\,=\,   \int_a^b \int_{B} u_f g \dx \dt 
      \,=\, \lim_{k\to\infty} \int_a^b \int_{B} u_{f_k} g_k \dx \dt \\
      &\,=\, \lim_{k\to\infty} \int_0^T \int_\Omega
      \bigl( c^{-2} \di_t^2 u_{f_k} - \Delta u_{f_k} + qu_{f_k} \bigr)
      v_{g_k} \dx \dt \\
      &\phantom{\,=\,}
      + \lim_{k\to\infty} \int_0^T \int_{\di\Omega}
      u_{f_k} \di_\bfnu v_{g_k} \ds_x \dt
      - \lim_{k\to\infty} \int_0^T \int_{\di\Omega}
      v_{g_k} \di_\bfnu u_{f_k} \ds_x \dt \\ 
      &\,=\, \lim_{k\to\infty} \bigl\langle {f_k} ,
      \di_{\bfnu} v_{g_k} \bigr\rangle_{L^2(\Gamma_T)} 
      \,=\, \bigl\langle f ,
      \di_{\bfnu}v_g \bigr\rangle_{L^2(\Gamma_T)} \,.
    \end{split}
  \end{equation*}
  This ends the proof. 
\end{proof}

\subsection{Localization in space}
\label{subsec:Localization_Space_Acoustic}
In Theorem~\ref{thm:LocalizedWavefunctions_acoustic} we establish the
existence aspects of solutions to \eqref{eq:IBVP_acoustic} that are
localized in space. 
Their construction will be discussed in
Corollary~\ref{cor:LocalizedWavefunctions_acoustic} below. 

\begin{theorem}
  \label{thm:LocalizedWavefunctions_acoustic}
  Let $D\Subset\Omega$ be open with Lipschitz boundary such that
  $\Omega_D:=\Omega\setminus\ol{D}$ is connected, and let 
  $B\subset\Omega$ be open with $B\not\subset D$. 
  For $0\leq a<b\leq T$ and $\dist_{\Omega_D}(\Omega_D, \Gamma)<b$, 
  there exists a sequence $\{f_k\}_{k\in\N}\subset L^2(\Gamma_T)$ such 
  that
  \begin{equation*}
    \|u_{f_k}\|_{L^2(B_{a,b})} \to \infty
    \qquad \text{and} \qquad
    \|u_{f_k}\|_{L^2(D_{T})} \to 0 
    \qquad \text{as } k\to\infty \,,
  \end{equation*}
  where $u_{f_k}$ denotes the solution to \eqref{eq:IBVP_acoustic} for
  $f=f_k$.
\end{theorem}

\begin{proof}
  We start with a brief sketch of the proof.
  Note that without loss of generality we may
  consider $B$ to be a sufficiently small open ball, say $B(z,\rho)$ for 
  some $\rho>0$, such that $B(z,\rho)\Subset\Omega_D$ and 
  $\widetilde\Omega:=\Omega\setminus(\ol{B(z,\rho)\cup D})$ is connected. 
  We may furthermore assume 
  $\dist_{\widetilde\Omega}(\widetilde\Omega, \Gamma)<b$ for such choice 
  of small $\rho>0$. 
  For all the discussion that follows, we simply write again $B$ in place 
  of~$B(z,\rho)$. 
  
  Defining the two
  operators~${\L_{B_{a,b}}:\, L^2(\Gamma_T) \to L^2(B_{a,b})}$ and
  $\L_{D_T}:\, L^2(\Gamma_T) \to L^2(D_T)$ as in~\eqref{eq:DefLab},
  after replacing $B_{a,b}$ by~${D_T}$ in this definition for the
  second one, we may reframe
  Theorem~\ref{thm:LocalizedWavefunctions_acoustic} and look for the
  existence of~$\{f_k\}_{k\in\N}\subset L^2(\Gamma_T)$ for which 
  \begin{equation}
    \label{eq:EquivForm_acoustic}
    \|\L_{B_{a,b}}[f_k]\|_{L^2(B_{a,b})} \to \infty
    \qquad \text{and} \qquad
    \|\L_{D_{T}}[f_k]\|_{L^2(D_{T})} \to 0
    \qquad \text{as } k\to\infty \,.
  \end{equation}
  In view of Lemma~\ref{lmm:Construction} in the appendix, we can
  alternatively prove the non-inclusion 
  \begin{equation}
    \label{eq:non-inclusion_acoustic}
    \ran \L^*_{B_{a,b}} \not\subseteq \ran \L^*_{D_{T}} \,,
  \end{equation}
  in order to establish \eqref{eq:EquivForm_acoustic}.
  We rely on a contrapositive argument to prove
  \eqref{eq:non-inclusion_acoustic}.
  For the sake of presentation, we divide the details of our proof
  in two steps.  

  \emph{Step I:} \;
  The adjoints of the operators $\L_{B_{a,b}}$ and $\L_{D_{T}}$
  are given by $\L^*_{B_{a,b}} :\, L^2(B_{a,b}) \to  L^2(\Gamma_T)$
  and $\L_{D_{T}}^*:\, L^2(D_T) \to L^2(\Gamma_T)$ as described in
  Lemma~\ref{eq:AdjointLab}, again after replacing $B_{a,b}$
  by~${D_T}$ in this definition for the second one. 
  Let us now assume that
  \begin{equation*}
    h \in \ran \L^*_{B_{a,b}} \cap \ran \L^*_{D_{T}} \,. 
  \end{equation*}
  Then, $h = \di_{\bfnu}v_{g_1} \vert_{\Gamma_T} =
  \di_{\bfnu}v_{g_2}\vert_{\Gamma_T}$, where $v_{g_i}$, $i=1,2$,
  solves the IBVP \eqref{eq:adjointIBVP_acoustic} for some
  $g_1 \in L^2(B_{a,b})$ and $g_2\in L^2(D_{T})$, respectively.
  Denoting
  $\widetilde\Omega := \Omega\setminus(\ol{B\cup D})$ and
  choosing $\delta>0$ small enough such that 
  $b-\delta>\max\{\dist_{\widetilde\Omega}(\widetilde\Omega, \Gamma), a\}$, we use a result on
  the unique continuation of Cauchy data for the wave equation to
  prove that 
  \begin{equation}
    \label{eq:same_wave-field_acoustic}
    v_{g_1} \,=\, v_{g_2}
    \qquad \text{in } \widetilde\Omega_{b-\delta,T} \,. 
  \end{equation}
  
  To see \eqref{eq:same_wave-field_acoustic}, we define
  $w : = v_{g_1} - v_{g_2}$ in $\widetilde\Omega_T$.
  We extend~$w$ from $\widetilde\Omega_T$ to
  $\widetilde\Omega_{2T}$ by zero, and with an abuse of notation we 
  also denote this extension by $w$. 
  Since $w|_{t=T} = \di_t w|_{t=T} = 0$ in $\widetilde\Omega$, we
  observe that this extension satisfies
  $w \in C([0,2T];H^1_0(\Omega))\cap C^1([0,2T];L^2(\Omega))$ and
  \begin{align*}
    c^{-2}(x)\di_t^2 w - \Delta w + q(x)w
    &\,=\, 0
    &&\text{in } \widetilde\Omega_{2T} \,, \\
    w \vert_{\Gamma_{2T}} \,=\, \di_\bfnu w\vert_{\Gamma_{2T}}
    &\,=\, 0
    &&\text{on } \Gamma_{2T} \,.
  \end{align*}
  In view of the vanishing of the Cauchy data of $w$ on
  $\Gamma_{2T}$, the global unique continuation principle for the wave
  equation~\cite[Thm~3.16]{KatKurLas2001} implies that 
  \begin{equation}
    \label{eq:UCP_acoustic}
    w \,=\, 0
    \qquad \text{in } \; \bigl\{ (x,t)\in\widetilde\Omega_{2T} 
    \;\big|\; \dist_{\widetilde\Omega}(x, \Gamma) \leq T-|t-T| \bigr\} \,.
  \end{equation}
  In conjunction with \eqref{eq:UCP_acoustic}, the choice of
  $\delta$ gives that  $w = 0$ in
  $\widetilde\Omega_{b-\delta,T}$,
  proving~\eqref{eq:same_wave-field_acoustic}. 
  Here we used the fact that $\dist_{\widetilde\Omega}(\widetilde\Omega, \Gamma)+\delta<b$.
  
  We now show that 
  \begin{equation}
    \label{eq:conseq_intersection_acoustic}
    v_{g_1} \,=\, 0
    \quad \text{in } \Omega\setminus\ol{B}\times(b-\delta,T) \,.
  \end{equation}
  Taking \eqref{eq:same_wave-field_acoustic} into consideration, we can
  define 
  \begin{equation*}
    v_{\mathrm{com}} \,:=\, \begin{cases}
      v_{g_2}
      &\text{in } B \times (b-\delta,T) \,, \\
      v_{g_1}
      &\text{in } D \times (b-\delta, T) \,, \\
      v_{g_1} = v_{g_2}
      &\text{in } \widetilde\Omega_{b-\delta,\, T} \,,
    \end{cases} 
  \end{equation*}
  which satisfies the homogeneous problem
  \begin{align*}
    c^{-2}(x) \di_t^2 v_{\mathrm{com}}
    - \Delta v_{\mathrm{com}}+ q(x) v_{\mathrm{com}}
    &\,=\, 0
    &&\text{in } \Omega_{b-\delta, T} \,, \\
    v_{\mathrm{com}}|_{t=T} \,=\, \di_t v_{\mathrm{com}}|_{t=T}
    &\,=\, 0
    &&\text{in }\Omega \,, \\
    v_{\text{com}}|_{(\di\Omega)_{b-\delta},T}
    &\,=\, 0
    &&\text{on } (\di\Omega)_{b-\delta,\,T} \,.
  \end{align*}
  The above problem, being well-posed, only admits the trivial
  solution.
  Therefore, we can conclude that $v_{\text{com}} = 0$ in
  $\Omega_{b-\delta,T}$, which establishes
  \eqref{eq:conseq_intersection_acoustic}. 

  \emph{Step II:} \;
  Next we show that there exists
  $g\in L^2(B_{a,b})$ such that $v_{g}$ solving
  \eqref{eq:adjointIBVP_acoustic} satisfies 
  \begin{equation}
    \label{eq:radiating_soln_acoustic}
    v_{g}\big|_{(\Omega\setminus\ol{B})\times(b-\delta,T)}
    \,\not\equiv\, 0 \,. 
  \end{equation}
  To this end, we employ Lemma~\ref{lmm:radiating_sources_acoustic}
  below for the choice $\tau=\delta$.
  As a consequence, we have $\widetilde{g}\in L^2(B_\delta)$ such that
  the solution $w_{\widetilde{g}}$ of \eqref{eq:forward_IBVP_body_source}
  satisfies $w_{\widetilde{g}}\not\equiv 0$ in 
  $(\Omega\setminus\ol{B})\times(0,\delta)$.
  Applying time-translation and time-reversal as introduced
  in~\eqref{eq:time-reversal_and_translation}, we define
  $g = \Rcal_T\Tcal_{T-b}[\widetilde{g}]$ in~$\Omega_T$ and notice
  that $\supp g \subseteq B_{a,b}$.
  Now let $v_{g}$ denote the associated solution of
  \eqref{eq:adjointIBVP_acoustic}. 
  Then $v_{g}=0$ in~$\Omega_{b,T}$ and
  $v_{g} \not\equiv 0$ in
  ${(\Omega\setminus\ol{B})\times(b-\delta,\,b)}$, 
  satisfying \eqref{eq:radiating_soln_acoustic}.
  It is immediate from the definition 
  that~$\di_\bfnu v_{g}\big\vert_{\Gamma_T}\in\ran\L^*_{B_{a,b}}$.
  However,  $\di_\bfnu
  v_{g}\big\vert_{\Gamma_T}\not\in\ran\L^*_{D_{T}}$.
  This is because, if~$\di_\bfnu
  v_{g}\big\vert_{\Gamma_T}\in\ran\L^*_{D_{T}}$,
  then we will have $v_{g}\equiv 0$ in
  $(\Omega\setminus\ol{B})\times(b-\delta,T)$ from
  \eqref{eq:conseq_intersection_acoustic} contradicting our choice
  of~$g$.
  In conclusion, we have proved~\eqref{eq:non-inclusion_acoustic},
  which ends the proof of
  Theorem~\ref{thm:LocalizedWavefunctions_acoustic}. 
\end{proof}

In Step~II of the proof of
Theorem~\ref{thm:LocalizedWavefunctions_acoustic} we have used the
following auxiliary result. 

\begin{lemma}
  \label{lmm:radiating_sources_acoustic}
  Consider an open subset $B \subset \Omega$ satisfying
  $\Omega\setminus\ol B \neq \emptyset$ and let
  $0 < \tau \leq T$.
  There exists a source~${g \in L^2(B_\tau)}$ for which
  $w_{g}|_{(\Omega\setminus\ol{B})\times(0,\tau]} \not\equiv 0$,
  where $w_{g} \in C([0,T];H^1_0(\Omega))\cap C^1([0,T];L^2(\Omega))$
  is the unique weak solution of the IBVP
  \begin{subequations}
    \label{eq:forward_IBVP_body_source}
    \begin{align}
      c^{-2}(x) \di_{t}^2 w_g - \Delta w_g + q(x) w_{g}
      &\,=\, g
      &&\text{in } \Omega_T \,, \\
      w_g|_{t=0} \,=\, \di_t w_g|_{t=0}
      &\,=\, 0
      &&\text{in } \Omega \,, \\
      w_g|_{(\di\Omega)_T}
      &\,=\, 0
      &&\text{on } (\di\Omega)_T \,.
    \end{align}
  \end{subequations}
  Moreover, such $g\in L^2(B_\tau)$ can be constructed.
\end{lemma}

\begin{proof}
  Denoting
  $\Mcal^{(\tau)}:=\{x\in\Omega\setminus\ol{B} \;|\;
  \dist_\Omega(x,\di B)<\tau\}$, we define the operator 
  \begin{equation}
    \label{eq:final_time_operator_acoustic}
    \T_\tau:\, L^2(B_\tau) \to L^2(\Mcal^{(\tau)}) \,,
    \quad g \mapsto w_g|_{\Mcal^{(\tau)}\times\{\tau\}} \,.
  \end{equation}
  Employing the unique continuation principle for wave
  equation~\cite[Thm.~3.16]{KatKurLas2001}, we can see 
  that~$\ran\T_\tau$ is dense in~$L^2(\Mcal^{(\tau)})$. 
  We skip discussing the details here as it follows from 
  adjusting approximate interior controllability for 
  wave equation; see \cite[Thm.~3.26]{LauLea2023}. 
  The amount of time required for such an approximate 
  controllability result is $2\tau$ as one aims to 
  control both $w_g|_{t=\tau}$ and~$\partial_tw_g|_{t=\tau}$. 
  However we are interested only in $w_g|_{t=\tau}$ 
  and therefore the time-length $\tau$ is sufficient 
  for density of the range of $\T_\tau$. 
  This can be proved by means of an odd reflection 
  argument which we showcase in 
  Lemma~\ref{lmm:radiating_sources_Maxwell_E} in the 
  context of electrodynamics.
  In conclusion, we can choose $g\in L^2(B_\tau)$ such that
  $w_{g}|_{t=\tau} \not\equiv 0$ in~$\Mcal^{(\tau)}$.
  Note that $w_{g}$ is continuous in time, implying
  $w_{g}|_{(\Omega\setminus\ol{B})\times(0,\tau]} \not\equiv 0$.

  Finally we discuss the construction of such a source $g$. 
  For this, we consider the problem of minimizing the Tikhonov
  functional 
  \begin{equation}
    \label{eq:min_func_acoustic}
    \mathbb{W}_\beta[g]
    \,:=\, \| \T_\tau[g]-\one_{\Mcal^{(\tau)}} \|^2_{L^2(\Mcal^{(\tau)})}
    +  \beta \|g\|_{L^2(B_\tau)}^2 \,,
    \qquad \beta>0 \,,
  \end{equation}
  where $\one_{\Mcal^{(\tau)}}$ denotes the characteristic function 
  on~$\Mcal^{(\tau)}$.
  With $\I$ representing the identity operator, we note that if $g_\beta$ denotes the minimizer for \eqref{eq:min_func_acoustic},
  then $g_\beta$ satisfies 
  \begin{equation}
    \label{eq:minimizers_acoustic}
    g_\beta
    \,=\,  \bigl( \T_\tau^*\T_\tau + \beta \I \bigr)^{-1}
    \T_\tau^*[\one_{\Mcal^{(\tau)}}] \,,
  \end{equation}
  and $ \lim_{\beta\to 0+} \T_\tau [g_\beta] = \one_{\Mcal^{(\tau)}}$.
  For a proof of \eqref{eq:minimizers_acoustic}, we refer the reader, 
  e.g., to~\cite[p.~52--55]{Han2017}. 
  This shows that $g_{\beta}$ given by
  \eqref{eq:minimizers_acoustic} is an example for the desired source
  $g$ in Lemma~\ref{lmm:radiating_sources_acoustic} when~$\beta>0$
  is sufficiently small. 
\end{proof}

Next we consider the construction of localized waves as in
Theorem~\ref{thm:LocalizedWavefunctions_acoustic}.   

\begin{corollary}
  \label{cor:LocalizedWavefunctions_acoustic}
  A sequence of Dirichlet data $\{f_k\}_{k\in\N}\subset L^2(\Gamma_T)$
  as in Theorem~\ref{thm:LocalizedWavefunctions_acoustic} can be
  explicitly constructed as 
  \begin{equation*}
    f_k
    \,:=\, \frac{1}{\|\,\sqrt[]{\J_k}[\xi]\|^{3/2}} \J_k[\xi]
    \quad \text{with} \quad \J_k
    \,:=\,  \bigl(\L^*_{D_{T}}\L_{D_{T}}
    + k^{-1}\I\bigr)^{-1} \,,
    \qquad k\in\N \,,
  \end{equation*}
  where, taking $\beta>0$ sufficiently small, we choose
  $0< \tau < \min\{b-a, b-\dist_{\Omega_D}(\Omega_D, \Gamma)\}$ and
  \begin{equation*}
    \xi
    \,:=\, \L^*_{B_{a,b}}
    \Bigl[ \Rcal_T\Tcal_{T-b}
    \bigl[ (\T_{\tau}^*\T_{\tau} + \beta \I)^{-1}
    \T_{\tau}^*[\one_{\Mcal^{(\tau)}}] \bigr] \Bigr] \,.
  \end{equation*}
  Here, the operators $\L_{B_{a,b}}$, $\L_{D_{T}}$, and
  $\T_{\tau}$ are defined in \eqref{eq:DefLab} and
  \eqref{eq:final_time_operator_acoustic}, respectively.  
\end{corollary}

\begin{proof}
  For the construction of localized waves, we make use of
  Lemma~\ref{lmm:Construction} with 
  $\Acal_1 = \L^*_{B_{a,b}}$, $\Acal_2 = \L^*_{D_{T}}$,
  and $\xi = \L^*_{B_{a,b}} [ \Rcal_T\Tcal_{T-b}[g] ]$ where,
  for sufficiently small $\beta>0$, we consider 
  \begin{equation*}
    g
    \,:=\, \bigl(\T_{\tau}^*\T_{\tau} + \beta \I\bigr)^{-1}
    \T_{\tau}^*[\one_{\Mcal^{(\tau)}}] \,,
  \end{equation*}
  with $0< \tau < \min\{b-a, b-\dist_{\Omega_D}(\Omega_D, \Gamma)\}$ and
  $\T_{\tau}$ defined in \eqref{eq:final_time_operator_acoustic}.
  The arguments in the Step~II of the proof of
  Theorem~\ref{thm:LocalizedWavefunctions_acoustic} yield 
  $\xi \notin \ran\L^*_{D_{T}}$.
  If 
  $\J_k :=  ( \L^*_{D_{T}} \L_{D_{T}}+ k^{-1} \I )^{-1}$
  then $\J_k$ is positive definite and therefore $\sqrt{\J_k}$ 
  makes sense for $k\in\N$.
  If we denote $\eta_k := \J_k[\xi]$,
  then~$\langle \xi , \eta_k\rangle
  = \langle \xi , \J_k[\xi]\rangle
  = \| \sqrt{\J_k}[\xi] \|^2$ for $k\in\N$. In view of
  Lemma~\ref{lmm:Construction}, the construction of localized wave 
  functions of Theorem~\ref{thm:LocalizedWavefunctions_acoustic}
  follows. 
\end{proof}

\begin{remark}
  \label{rem:Finite_speed}
  Due to the finite speed of propagation, we note that
  $\dist_\Omega(\Omega,\Gamma)$ is the minimum time needed for the 
  waves to reach every part of $\Omega$.
  This can be justified by a standard domain of dependence argument. 
  See, for instance, \cite[Thm.~2.47]{KatKurLas2001}. 
  However, our assumption~$\dist_{\Omega_D}(\Omega_D,\Gamma) < b$ in Theorem~\ref{thm:LocalizedWavefunctions_acoustic} is natural since the 
  localized solutions should avoid the subset $D$ for all time, while 
  focusing on the cylinder $B_{a,b}$ where $B$ may be an arbitrary open set 
  from $\Omega_D$.
  Also notice that the choice of $a$ did not play any role in
  Theorem~\ref{thm:LocalizedWavefunctions_acoustic}.
  Therefore we can replace~$a$ by $a_k$ in 
  Theorem~\ref{thm:LocalizedWavefunctions_acoustic} where 
  $0 \le a_k <b, \, k \in \N$ with $a_k \to b$ as $k\to\infty$. 
  This implies that the high energy part of localized waves can
  correspond to any small time-interval (around~$b$) of our choice
  whenever the condition $b>\dist_{\Omega_D}(\Omega_D,\Gamma)$ is
  fulfilled.~\hfill$\lozenge$ 
\end{remark}

\begin{remark}
  \label{rem:smooth_localized_potentials}
  Theorem~\ref{thm:LocalizedWavefunctions_acoustic} is also valid for
  a sequence from any dense subset of $L^2(\Gamma_T)$.
  This immediately implies that the sequence of Dirichlet data in
  Theorem~\ref{thm:LocalizedWavefunctions_acoustic} can be chosen to
  be $C^\infty$-smooth.
  As a consequence, one can also improve the smoothness of localized
  waves.
  We further add that one can prove results similar to
  Theorem~\ref{thm:LocalizedWavefunctions_acoustic} using interior
  sources (in place of boundary sources) which are supported in
  $\omega_T$ where $\omega\subseteq\Omega$ is an arbitrary open set
  and $T>0$ is sufficiently large.~\hfill$\lozenge$
\end{remark}

\begin{remark}
  Our arguments to construct localized solutions mainly use Tataru's 
  sharp global unique continuation result \cite{Tat95} for the 
  wave equation which holds even for more general hyperbolic operators
  admitting time-independent (or, time-analytic) coefficients such as 
  \begin{equation*}
    c^{-2}(x)\di^2_t
    - \div\left(A(x)\nabla_x \right)
    + B(x)\cdot\nabla_{x} + q(x)
  \end{equation*}
  where $A$ and $c$ denote some $C^1$-smooth positive
  definite matrix-valued and non-negative functions,
  respectively.
  Furthermore $B(\cdot)$ and $q(\cdot)$ represent some bounded vector
  field and scalar function, respectively.  
  Therefore, Theorem~\ref{thm:LocalizedWavefunctions_acoustic} can 
  also be established for such general operators.~\hfill$\lozenge$
\end{remark}

\subsection{Localization in time}
\label{subsec:Localization_Time_Acoustic} 
Similarly we can construct boundary sources for which the solutions to
\eqref{eq:IBVP_acoustic} concentrate on any given open set
$B\subset\Omega$ at a sufficiently large time. 

\begin{theorem}
  \label{thm:localizedSolution__same_base_acoustic}
  Consider an open subset $B\Subset\Omega$ with Lipschitz boundary such that
  $\Omega_B:=\Omega\setminus\ol{B}$ is connected, and let $[a,b]$ and~$[c,d]$ be
  two subintervals of~$[0,T]$ with $[a,b]\cap[c,d]=\emptyset$ such
  that
  \begin{itemize}
  \item[(I)] $\dist_{\Omega}(\Omega, \Gamma) < |b-d|$ \hspace*{13.5em} if $c<a$\,,
  \item[(II)] $\dist_{\Omega_B}(\Omega_B, \Gamma) < b$ \; and \;
        $\dist_B(B,\di B) < |b-c| / 2$ \qquad if $c>a$\,.
  \end{itemize}
  Then, there exists a sequence $\{f_k\}_{k\in\N} \subset L^2(\Gamma_T)$ such that 
  \begin{equation*}
    \|u_{f_k}\|_{L^2(B_{a,b})} \to \infty
    \qquad \text{and} \qquad
    \|u_{f_k}\|_{L^2(B_{c,d})} \to 0
    \qquad \text{as } k\to\infty \,,
  \end{equation*}
  where $u_{f_k}$ denotes the solution to \eqref{eq:IBVP_acoustic} for
  $f=f_k$.
  Moreover, such boundary sources $\{f_k\}_{k\in\N}$ can be
  constructed. 
\end{theorem}

\begin{proof}
  We present the proof of
  Theorem~\ref{thm:localizedSolution__same_base_acoustic} in two separate
  cases. 
  
  \emph{Case I: ($c<a$)} \;
  The intervals $[c,d]$ and $[a,b]$ being disjoint, we also have
  $d<a$.
  Moreover, it is sufficient to consider $f_k$ belonging to
  $L^2(\Gamma_b)$. 
  If we can establish the existence of~$f\in L^2(\Gamma_{d,b})$
  such that the solution $u_{f}$ to \eqref{eq:IBVP_acoustic}
  satisfies $u_{f}\not\equiv 0$ in $B_{a,b}$, then we can consider
  $f_k = kf$ for~$k\in\N$ in $\Gamma_T$. 
  With such consideration, we have $u_{f_k}= k\,u_{f}$ in
  $\Omega_T$ and therefore 
  \begin{equation*}
    \|u_{f_k}\|_{L^2(B_{a,b})} \to \infty
    \quad \text{as } k\to\infty
    \qquad \text{and} \qquad
    \|u_{f_k}\|_{L^2(B_{c,d})} \,=\, 0
    \quad \text{for all } k\in\N \,.
  \end{equation*}

  Now we argue to prove the existence and construction of
  $f\in L^2(\Gamma_{d,T})$ for which $u_{f}\not\equiv 0$ in
  $B_{a,b}$.
  Note that, $u_{f}= 0$ in $\Omega_d$.
  Invoking a time-translation, it therefore reduces to showing
  $u_{f}\not\equiv 0$ in $B_{a-d,b-d}$ for some
  $f\in L^2(\Gamma_{b-d})$.
  At this point, we take $\tau = b-d$ in
  Lemma~\ref{lmm:radiating_boundary_sources_acoustic} below to
  guarantee the existence of such $f\in L^2(\Gamma_{b-d})$.
  Here we use that $\dist_\Omega(\Omega,\Gamma) < |b-d|$. 
  
  Also, we can consider a minimization argument similar to that of
  Lemma~\ref{lmm:radiating_sources_acoustic} to construct the
  boundary source 
  \begin{equation*}
    f
    \,=\, (\P_\tau^*\P_\tau + \beta \I )^{-1} \P_\tau^*[\one_{B}] \,,
  \end{equation*}
  with $\P_\tau$ defined in \eqref{eq:DefOperatorPtau} below and
  $\beta>0$ sufficiently small such that
  $u_{f}|_{t=b-d} \not\equiv 0$ in $B$ at time~$t=b-d$.
  Now the continuity of $u_f$ with respect to time ensures that
  $u_{f} \not\equiv 0$ in $B_{a-d,b-d}$ completing our argument
  for Case~I. 

  \emph{Case II: ($a<c$)} \;
  We have $b < c$ as well since $[c,d]$ and $[a,b]$
  are disjoint.
  Here we aim to establish the following relation 
  \begin{equation}
    \label{eq:non-inclusion_same_base}
    \ran\L^*_{B_{a,b}} \cap \ran\L^*_{B_{c,d}} \,=\, \{0\} 
  \end{equation}
  with $\L_{B_{a,b}}$ and $\L_{B_{c,d}}$ defined analogous to
  \eqref{eq:DefLab}.
  In view of the relation \eqref{eq:non-inclusion_same_base}, we may
  refer to Lemma~\ref{lmm:Construction} for the existence of
  boundary sources $f_k$ needed in
  Theorem~\ref{thm:localizedSolution__same_base_acoustic}.
  
  To prove \eqref{eq:non-inclusion_same_base}, we rely on a
  contrapositive argument as done in
  Theorem~\ref{thm:LocalizedWavefunctions_acoustic}.
  Suppose that $h\in\ran\L^*_{B_{a,b}}\cap\ran\L^*_{B_{c,d}}$.
  This implies
  $h=\di_\bfnu v_{g_1}|_{\Gamma_T}=\di_\bfnu v_{g_2}|_{\Gamma_T}$
  where $v_{g_1}$ and $v_{g_2}$ solve the adjoint problem
  \eqref{eq:adjointIBVP_acoustic} corresponding to sources
  $g_1\in L^2(B_{a,b})$ and $g_2\in L^2(B_{c,d})$, respectively.
  Since $g_1=g_2=0$ in $B_{d,T}$, we obtain
  from~\eqref{eq:adjointIBVP_acoustic} that $v_{g_1}=v_{g_2}=0$ in 
  $\Omega_{d,T}$.
  In particular, ${h|_{\Gamma_{d,T}}=0}$. 
  We define 
  $w := v_{g_1}-v_{g_2}$ in~$\Omega_d$. 
  Extending $w$ from $\Omega_d$ to $\Omega_{2d}$ by zero, and denoting
  this extension again by $w$, we find
  that~$w \in C([0,2d];H^1_0(\Omega))\cap C^1([0,2d];L^2(\Omega))$ and
  \begin{align*}
    c^{-2}(x)\di_t^2w - \Delta w + q(x) w
    &\,=\, 0
    &&\text{in } (\Omega_B)_{2d} \,, \\
    w|_{t=d} \,=\, \di_t w|_{t=d}
    &\,=\, 0
    &&\text{in } \Omega \,, \\
    w|_{\Gamma_{2d}} \,=\, \di_\bfnu w|_{\Gamma_{2d}} 
    &\,=\, 0
    &&\text{on } \Gamma_{2d} \,.
  \end{align*}
  Invoking the unique continuation 
  argument~\cite[Thm.~3.16]{KatKurLas2001}, we see that 
  \begin{equation*}
    w \,=\, 0 \qquad
    \text{in } \;
    \bigl\{ (x,t) \in (\Omega_B)_{2d}
    \;\big|\; \dist_{\Omega_B}(x,\Gamma) \leq d-|t-d| \bigr\} \,.
  \end{equation*}
  In view of our assumption $\dist_{\Omega_B}(\Omega_B,\Gamma) < b$, we
  obtain~$w(x,t)=0$,   
  i.e., $v_{g_1}=v_{g_2}$ in $(\Omega_B)_{b,d}$.
  Since $g_1(\cdot,t)=0$ for almost every $t\in(b,T)$, we also have
  $v_{g_1}=0$ in $\Omega_{b,d}$.
  In particular, $h|_{\Gamma_{b,d}}=0$. 
  Furthermore, this implies $v_{g_2}=0$ in $(\Omega_B)_{b,c}$. 
  Again we use the unique continuation
  argument~\cite[Thm.~3.16]{KatKurLas2001} to see that 
  \begin{equation*}
    v_{g_2} \,=\, 0 \qquad
    \text{in } \; \Bigl\{ (x,t) \in B_{b,c}
    \;\Big|\; \dist_B(x,\di B) < \frac{c-b}{2}
    + \Bigl| t - \frac{c+b}{2}\Bigr| \Bigr\} \,,
  \end{equation*}
  which, due to the assumption $\dist_B(B,\di B)< |b-c| / 2$, yields
  $v_{g_2} = \di_t v_{g_2} = 0$ in the
  slice~${B\times\{\frac{c+b}{2}\}}$.
  Together with what we have seen before, we conclude that
  $v_{g_2} = \di_t v_{g_2}=0$ on $\Omega\times\{ \frac{c+b}{2} \}$.
  Since~$g_2(\cdot,t)=0$ for almost every $t\in(0,\frac{c+b}{2})$,
  we use the uniqueness of solutions to the adjoint problem to
  conclude~$v_{g_2}=0$ in $\Omega_{\frac{c+b}{2}}$ which further gives
  $h = 0$ on $\Gamma_{b}$.
  Combining this with~$h|_{\Gamma_{b,d}}=0$ and
  $h|_{\Gamma_{d,T}}=0$, our proof for
  \eqref{eq:non-inclusion_same_base} is complete.  

  For the construction of $f_k$, we need to first obtain any
  non-zero element from $\ran\L^*_{B_{a,b}}$, say~$\xi$, and then
  appeal to Lemma~\ref{lmm:Construction} with the choice 
  \begin{equation*}
    f_k
    \,:=\, \frac{1}{\|\sqrt{\J_k}[\xi]\|^{3/2}} \J_k[\xi]
    \quad \text{with} \quad
    \J_k \,:=\,  \bigl(\L^*_{B_{c,d}}\L_{B_{c,d}}+ k^{-1}\I\bigr)^{-1} \,,
    \qquad k\in\N \,.
  \end{equation*}
  Now, to find $0\neq \xi \in\ran\L^*_{B_{a,b}}$, we may argue as in
  Step~II of Theorem~\ref{thm:LocalizedWavefunctions_acoustic}, which in
  turn relies on Lemma~\ref{lmm:radiating_sources_acoustic}.
  For brevity, we briefly discuss the steps.
  We first choose $\delta>0$
  satisfying~$\dist_{\Omega_B}(\Omega_B,\Gamma) + \delta < b$, then we consider
  \begin{equation*}
    g
    \,:=\, \bigl( \T_{\delta}^*\T_{\delta}
    + \beta\I \bigr)^{-1}
    \T_{\delta}^*[\one_{\Mcal^{(\delta)}}]
  \end{equation*}
  for $\beta>0$ sufficiently small, where $\T_{\delta}$ is
  defined in \eqref{eq:final_time_operator_acoustic} for
  $\tau=\delta$.
  We observe that~$\xi\not\equiv 0$ in~$\Gamma_T$ where  
  \begin{equation}
    \label{eq:radiating_data_same_base}
    \xi
    \,:=\, \L^*_{B_{a,b}} \bigl[ \Rcal_T\Tcal_{T-b}[g] \bigr]
  \end{equation}
  If $\xi \equiv 0$ on $\Gamma_{T}$, then we can again use unique
  continuation (see~\cite[Thm.~3.16]{KatKurLas2001}) to claim
  that~$v_{\Rcal_T\Tcal_{T-b}[g]}$ solving~\eqref{eq:adjointIBVP_acoustic}
  satisfies the vanishing condition 
  \begin{equation*}
    v_{\Rcal_T\Tcal_{T-b}[g]}(x,t) \,=\, 0 \qquad
    \text{in } \Omega_B\times (b-\delta,b) \,,
  \end{equation*}
  which contradicts our choice of $g$ implying $\xi$ defined in
  \eqref{eq:radiating_data_same_base} does not vanish on all
  of~$\Gamma_T$.  
\end{proof}

In Case~I of the proof of
Theorem~\ref{thm:localizedSolution__same_base_acoustic} we used the
following auxiliary result. 

\begin{lemma}
  \label{lmm:radiating_boundary_sources_acoustic}
  For $\tau>\dist_\Omega(\Omega,\Gamma)$ and $B\subseteq \Omega$ open, we
  can construct $f \in L^2(\Gamma_{\tau})$ such
  that~${u_{f}|_{t=\tau} \not\equiv 0}$ in $B$, where $u_{f}$ denotes
  the solution to \eqref{eq:IBVP_acoustic}. 
\end{lemma}

We omit the proof of
Lemma~\ref{lmm:radiating_boundary_sources_acoustic}, as the arguments 
are similar to that of Lemma~\ref{lmm:radiating_sources_acoustic}.
It relies on minimizing a Tikhonov functional along with a density
result which implies $\ol{\ran\P_{\tau}} = L^2(\Omega)$, with
$\P_{\tau}$ denoting the mapping 
\begin{equation}
  \label{eq:DefOperatorPtau}
  \P_{\tau}:\, L^2(\Gamma_{\tau}) \to L^2(\Omega) \,,
  \quad f \mapsto u_f|_{t=\tau} \,,
\end{equation}
and $u_f$ solving the IBVP \eqref{eq:IBVP_acoustic}. 
Once again, we underline to the point that the density of the range 
of $\P_{\tau}$ requires adjusting the standard boundary approximate 
controllability result for the wave equation given by, for instance, 
\cite[Thm.~4.28]{KatKurLas2001}. 
In this regard, we can use an odd reflection argument as discussed in 
Lemma~\ref{lmm:radiating_sources_Maxwell_E} below.

\section{Localized solutions for the Maxwell equations.}
\label{sec:MaxwellEquations}
Next we discuss the existence and construction of localized solutions 
of the IBVP for Maxwell's equations~\eqref{eq:IBVP_Maxwell}. 
As in Section~\ref{sec:WaveEquation}, our proofs rely on an analysis 
of the range spaces of the adjoints of certain restricted solution
operators. 
We start by introducing these operators and identifying their adjoints. 

For any bounded open subset $B\subset\Omega$ and $0\leq a< b\leq T$,
we define 
\begin{subequations}
  \label{eq:DefLabEabHab_Maxwell}
  \begin{align}
    \L_{B_{a,b}}&:\, \Hcal^1_0([0,T]; \HminhalfdivGamma)
                  \to \bfL^2(B_{a,b})^2 \,, 
    &&\bff \mapsto
       \bigl( \bfE_\bff\big|_{B_{a,b}},\bfH_\bff\big|_{B_{a,b}} \bigr) \,,  \\
    \E_{B_{a,b}}&:\, \Hcal^1_0([0,T]; \HminhalfdivGamma)
                  \to \bfL^2(B_{a,b}) \,, 
    &&\bff \mapsto \bfE_\bff\big|_{B_{a,b}} \,, \\
    \H_{B_{a,b}}&:\, \Hcal^1_0([0,T]; \HminhalfdivGamma)
                  \to \bfL^2(B_{a,b}) \,, 
    &&\bff \mapsto\bfH_\bff\big|_{B_{a,b}} \,,
  \end{align}
\end{subequations}
where $(\bfE_\bff,\bfH_\bff)$ denotes the weak solution
of~\eqref{eq:IBVP_Maxwell}.
In the following lemma we identify the adjoints of the three operators
from~\eqref{eq:DefLabEabHab_Maxwell}.
Here, we consider an inner product in $\bfL^2(B_{a,b})^2$, 
which is equivalent to the standard inner product.
For
$\bfJ=(\bfj_1,\bfj_2), \bfK=(\bfk_1,\bfk_2)
\in \bfL^2(B_{a,b})^2$, we define 
\begin{equation*}
  \bigl\langle \bfJ , \bfK \bigr\rangle_{\eps,\mu}
  \,:=\, \int_{B_{a,b}} \eps(x) \bfj_1(x,t)
  \cdot \bfk_1(x,t) \dx\dt
  + \int_{B_{a,b}} \mu(x) \bfj_2(x,t)
  \cdot \bfk_2(x,t) \dx\dt
\end{equation*}
We further introduce two equivalent inner products for
$\bfL^2(B_{a,b})$ given by 
\begin{equation}
  \label{eq:DefEpsMuScalarProduct}
  \langle \bfj , \bfk \rangle_\eps
  \,:=\, \int_{B_{a,b}} \eps(x) \bfj(x,t) \cdot \bfk(x,t) \dx\dt \,, \quad
  \langle \bfj , \bfk \rangle_\mu
  \,:=\, \int_{B_{a,b}} \mu(x) \bfj(x,t) \cdot \bfk(x,t) \dx\dt
\end{equation}
for $\bfj,\bfk \in \bfL^2(B_{a,b})$. 
The inner products $\langle \cdot,\cdot\rangle_\eps$ and
$\langle\cdot,\cdot \rangle_\mu$ are used for the definition
of
$\E^*_{B_{a,b}}$ and~$\H^*_{B_{a,b}}$, respectively.
In addition to \eqref{eq:time-reversal_and_translation}, we introduce
the notation 
\begin{equation*}
  \Scal_\tau[g](\ph,t) \,:=\, \int_\tau^t g(\ph,s) \ds \,,
\end{equation*}
which, with respect to the time-level $t=\tau\in\R$, represents the
time-integral operator.

\begin{lemma}
  \label{lmm:AdjointLabEabHab_Maxwell}
  The adjoints of the operators $\L_{B_{a,b}}$, $\E_{B_{a,b}}$, and
  $\H_{B_{a,b}}$ from~\eqref{eq:DefLabEabHab_Maxwell} are given by
  \begin{subequations}
    \begin{align}
      \L_{B_{a,b}}^* :\, \bfL^2(B_{a,b})^2
      \to \Hcal^1_0([0,T]; \HminhalfdivGamma^*) \,,  \quad
      & \bfJ \mapsto - \Scal_0 \bigl[
        \pi_\tau[\widetilde{\bfH}_{\bfJ}]\big|_{\Gamma_T} \bigr] \,, \\
      \E_{B_{a,b}}^* :\, \bfL^2(B_{a,b})
      \to \Hcal^1_0([0,T]; \HminhalfdivGamma^*) \,,  \quad
      &\bfj_1 \mapsto - \Scal_0 \bigl[
        \pi_\tau[\widetilde{\bfH}_{(\bfj_1,0)}]\big|_{\Gamma_T} \bigr] \,, \\
      \H_{B_{a,b}}^* :\, \bfL^2(B_{a,b})
      \to \Hcal^1_0([0,T]; \HminhalfdivGamma^*) \,, \quad
      &\bfj_2 \mapsto - \Scal_0 \bigl[
        \pi_\tau[\widetilde{\bfH}_{(0,\bfj_2)}]\big|_{\Gamma_T} \bigr] \,,
    \end{align}
  \end{subequations}
  where $(\widetilde{\bfE}_\bfJ,\,\widetilde{\bfH}_\bfJ) 
  \in C([0,T]; \bfH_0(\curl;\Omega)\times\bfH(\curl;\Omega))
  \cap C^1([0,T]; \bfL^2(\Omega)^2)$ denotes, for
  any $\bfJ=(\bfj_1,\,\bfj_2)\in \bfL^2(B_{a,b})^2$, the unique weak
  solution of the IBVP
  \begin{subequations}
    \label{eq:adjointIBVP_Maxwell}
    \begin{align}
      \eps(x) \di_t\widetilde{\bfE}_\bfJ - \curl\widetilde{\bfH}_\bfJ
      &\,=\, \eps(x) \Scal_T[\bfj_1]
      &&\text{in } \Omega_T \,, \\
      \mu(x) \di_t\widetilde{\bfH}_\bfJ + \curl\widetilde{\bfE}_\bfJ
      &\,=\, \mu(x) \Scal_T[\bfj_2]
      &&\text{in } \Omega_T \,, \\
      \widetilde{\bfE}_\bfJ|_{t=T} \,=\, \widetilde{\bfH}_\bfJ|_{t=T}
      &\,=\, 0
      &&\text{in } \Omega \,, \\
      \bfnu\times\widetilde{\bfE}_\bfJ|_{(\di\Omega)_{T}}
      &\,=\, 0 
      &&\text{on } (\di\Omega)_{T} \,. 
    \end{align}
  \end{subequations}
\end{lemma}

\begin{proof}
  To study the well-posedness of \eqref{eq:adjointIBVP_Maxwell} 
  we consider a time-reversal to convert the backward
  problem into a forward problem. 
  Then, observing that the source terms
  $\eps(x)\Scal_T[\bfj_1]$ and $\mu(x)\Scal_T[\bfj_2]$
  in~\eqref{eq:adjointIBVP_Maxwell} 
  belong to $H^1((0,T);L^2(\Omega))$, we can
  apply~\cite[Thm.~5.3]{KirRie2016} to see that 
  the~IBVP~\eqref{eq:adjointIBVP_Maxwell} has a unique weak solution
  $(\widetilde{\bfE}_\bfJ,\widetilde{\bfH}_\bfJ) \in
  C([0,T]; \bfH_0(\curl;\Omega)\times\bfH(\curl;\Omega))
  \cap C^1([0,T]; \bfL^2(\Omega)^2)$. 
  
  Next we discuss the characterization of the adjoint operator
  $\L^*_{B_{a,b}}$ and note that the characterizations of
  $\E_{B_{a,b}}^*$ and~$\H^*_{B_{a,b}}$ can be proved similarly. 
  Considering $\bff\in\Hcal^1_0([0,T]; \HminhalfdivGamma)$ and
  $(\bfj_1,\bfj_2)=\bfJ\in \bfL^2(B_{a,b})^2$, we denote the
  associated weak solutions to \eqref{eq:IBVP_Maxwell} and
  \eqref{eq:adjointIBVP_Maxwell} by~$(\bfE_{\bff},\bfH_{\bff})$ and
  $(\widetilde{\bfE}_\bfJ,\widetilde{\bfH}_\bfJ)$, respectively.
  Recalling from Proposition~\ref{pro:WellPosedness_IBVP_Maxwell}
  that $(\widetilde{\bfE}_\bfJ,\widetilde{\bfH}_\bfJ)$ is just
  continuous in time, we choose smooth approximations
  $\{\bff_k\}_{k\in\N} \subset C_c^\infty((0,T]; \HminhalfdivGamma)$ such
  that $\bff_k \to \bff$ in $\Hcal^1_0([0,T]; \HminhalfdivGamma)$.
  Accordingly, we denote by~$(\bfE_{\bff_k},\bfH_{\bff_k})$ the
  associated weak solution to~\eqref{eq:IBVP_Maxwell}. 
  The continuous dependence of solutions to~\eqref{eq:IBVP_Maxwell}
  on the data (see~\cite[Thm.~2]{Ant2025}) shows that
  $(\bfE_{\bff_k},\bfH_{\bff_k}) \to (\bfE_{\bff},\bfH_{\bff})$ in
  $C([0,T]; \bfL^2(\Omega)^2)$.
  Furthermore, we have~$(\bfE_{\bff_k},\bfH_{\bff_k}) \in
  C^\infty([0,T]; \bfH_0(\curl;\Omega)\times\bfH(\curl;\Omega))$
  (see~\cite[Thm.~2.6]{KirRie2016}) along
  with the relations $\di_t^p\bfE_{\bff_k} =
  \bfE_{\di_t^p\bff_k}$ and $\di_t^p\bfH_{\bff_k} =
  \bfH_{\di_t^p\bff_k}$ for any $k,p\in\N$.
  Here the weak solutions to \eqref{eq:IBVP_Maxwell} corresponding
  to boundary data $\di_t^p\bff_k$ are denoted by
  $(\bfE_{\di_t^p\bff_k}, \bfH_{\di_t^p\bff_k})$. 
  
  Now we use an integration by parts to compute 
  \begin{align*}
    &\bigl\langle \L_{B_{a,b}}[\bff] , \bfJ \bigr\rangle_{\eps,\mu}
      \,=\, \int_a^b \int_B \bigl(
      \eps(x) \bfE_\bff \cdot \bfj_1
      + \mu(x) \bfH_\bff \cdot \bfj_2 \bigr) \dx \dt \\
    &\,=\, \lim_{k\to\infty} \int_0^T \int_{\Omega}
      \bigl( \eps(x) \bfE_{\bff_k} \cdot \bfj_{1}
      + \mu(x) \bfH_{\bff_k} \cdot \bfj_2 \bigr) \dx \dt \\
    &\,=\, - \lim_{k\to\infty} \int_0^T \int_{\Omega}
      \bigl( \bfE_{\bff'_k}
      \cdot (\eps(x) \Scal_T[\bfj_1])
      + \bfH_{\bff'_k}
      \cdot (\mu(x) \Scal_T[\bfj_2]) \bigr) \dx \dt \\ 
    &\,=\, - \lim_{k\to\infty} \int_0^T \int_{\Omega}
      \Bigl( \bfE_{\bff'_k}
      \cdot \bigl( \eps(x) \di_t\widetilde{\bfE}_{\bfJ}
      - \curl\widetilde{\bfH}_{\bfJ} \bigr)
      + \bfH_{\bff'_k}
      \cdot \bigl( \mu(x) \di_t\widetilde{\bfH}_{\bfJ}
      + \curl\widetilde{\bfE}_{\bfJ} \bigr) \Bigr) \dx \dt \\ 
    &\,=\, \lim_{k\to\infty} \int_0^T \int_{\Omega}
      \Bigl( \widetilde{\bfE}_{\bfJ}
      \cdot \bigl(\eps(x) \di_t\bfE_{\bff'_k}
      - \curl\bfH_{\bff'_k} \bigr)
      + \widetilde{\bfH}_{\bfJ}
      \cdot \bigl( \mu(x)\di_t\bfH_{\bff'_k}
      + \curl\bfE_{\bff'_k} \bigr) \Bigr) \dx \dt \\ 
    &\phantom{\,=\,}
      - \lim_{k\to\infty}  \int_0^T \int_{\di\Omega}
      \gamma_\tau[\bfE_{\bff'_k}]
      \cdot \pi_\tau[\widetilde{\bfH}_{\bfJ}] \ds_x \dt \\ 
    &\,=\, - \lim_{k\to\infty} \int_0^T \int_{\Gamma}
      \bff'_k \cdot \pi_\tau[\widetilde{\bfH}_{\bfJ}] \ds_x \dt 
      \,=\, - \int_0^T \int_{\Gamma}
      \bigl\langle \di_t \bff , 
      \di_t \Scal_0\bigl[
      \pi_\tau[\widetilde{\bfH}_\bfJ] \bigr] \bigr\rangle \ds_x \dt \\
    &\,=\, \bigl\langle \bff , 
      - \Scal_0\bigl[ \pi_\tau[\widetilde{\bfH}_\bfJ] \bigr]
      \bigr\rangle_{\Hcal^1_0([0,T]; \HminhalfdivGamma)} \,, 
  \end{align*}
  which gives the characterization of $\L^*_{B_{a,b}}$.
\end{proof}

The reason for not using the standard inner product on
$\bfL^2(B_{a,b})^2$ in the definition of the adjoint operator in
Lemma~\ref{lmm:AdjointLabEabHab_Maxwell} stems from the following
observation, which will be used later in the proof of Theorem~\ref{thm:localizedSolution__same_base_Maxwell}.

\begin{remark}
  \label{rem:InvariantDefinitions}
  We recall the Helmholtz decompositions 
  \begin{equation*}
    \bfL^2(\Omega)
    \,=\, \{ \bfV \in \bfL^2(\Omega) \;|\; \div(\eps\bfV) = 0 \}
    \oplus \nabla H^1_0(\Omega) 
    \,=\, \{ \bfV \in \bfL^2(\Omega) \;|\; \div(\mu\bfV) = 0 \}
    \oplus \nabla H^1_0(\Omega) 
  \end{equation*}
  (see, e.g., \cite[Thm.~4.23]{KirHet2015}), where the divergence has
  to be understood in weak sense.
  These are orthogonal with respect to the inner products 
  $\langle\cdot,\cdot,\rangle_{\eps}$ and 
  $\langle\cdot,\cdot,\rangle_{\mu}$, respectively.
  Accordingly, we can decompose $\bfJ=(\bfj_1,\bfj_1)\in\bfL^2(B_{a,b})^2$
  into 
  \begin{align*}
    \bfj_1
    &\,=\, \widetilde\bfj_1 + \nabla\phi_1
    &&\text{with } \div\bigl(\eps(x)\widetilde\bfj_1(t)\bigr) = 0
       \text{ in } \Omega \text{ and } \phi_1(t)\in H^1_0(\Omega) \,, \\
    \bfj_2
    &\,=\, \widetilde\bfj_2 + \nabla\phi_2
    &&\text{with } \div\bigl(\mu(x)\widetilde\bfj_2(t)\bigr) = 0
       \text{ in } \Omega \text{ and } \phi_2(t)\in H^1_0(\Omega) \,,
  \end{align*}
  for a.e.\ $t\in(a,b)$.
  Denoting $\widetilde\bfJ:=(\widetilde\bfj_1,\widetilde\bfj_2)$ and
  using the fact that the solution $(\bfE_\bff,\bfH_\bff)$ to 
  \eqref{eq:IBVP_Maxwell} satisfies
  \begin{equation*}
    \div(\eps(x)\bfE_\bff(t))
    \,=\, \div(\mu(x)\bfH_\bff(t))
    \,=\, 0 \qquad \text{for all } t\in [0,T] \,,
  \end{equation*}
  we obtain that
  $\L^*_{B_{a,b}}[\bfJ] = \L^*_{\Omega_{a,b}}[\widetilde\bfJ]$.
  We can draw similar conclusions for $\E^*_{B_{a,b}}$ and
  $\H^*_{B_{a,b}}$.~\hfill$\lozenge$ 
\end{remark}

\subsection{Localization in space}
\label{subsec:Localization_Space_Maxwell}
In Theorem~\ref{thm:LocalizedWavefunctions_Maxwell} we establish the
existence aspects of solutions to \eqref{eq:IBVP_Maxwell} that are
localized in space. 
Their construction will be discussed in
Corollary~\ref{cor:LocalizedWavefunctions_Maxwell} below. 
Since the unique continuation principles from \cite{EINT2002,Ell2003}
require $C^2$-smooth electric permittivity $\eps$ and magnetic
permeability $\mu$, we add this assumption in the following theorems. 

\begin{theorem}
  \label{thm:LocalizedWavefunctions_Maxwell}
  Suppose, in addition to our previous assumptions, that 
  $\eps,\mu\in C^2(\ol{\Omega})$.
  Let $D\Subset\Omega$ be open with Lipschitz boundary such that $\Omega_D:= \Omega\setminus\ol{D}$ is
  connected, and let $B\subset\Omega$  be open with $B\not\subset D$. 
  For $0\leq a< b\leq T$ and $\dist_{\Omega_D}(\Omega_D, \Gamma)<b$ there exists a
  sequence
  $\{\bff_k\}_{k\in\N} \subset \Hcal^1_0([0,T]; \HminhalfdivGamma)$
  such that 
  \begin{equation}
    \label{eq:LocWavfun_Maxwell_E}
    \| \bfE_{\bff_k} \|_{\bfL^2(B_{a,b})} \to \infty
    \qquad \text{and} \qquad
    \| \bfE_{\bff_k} \|_{\bfL^2(D_{T})}
    + \| \bfH_{\bff_k} \|_{\bfL^2(D_{T})} \to 0
    \qquad \text{as } k\to\infty \,,
  \end{equation}
  where $(\bfE_{\bff_k}, \bfH_{\bff_k})$ denotes the solution to
  \eqref{eq:IBVP_Maxwell} for $\bff=\bff_k$. 

  Furthermore, there exists a sequence
  $\widetilde{\bff}_k\in \Hcal^1_0([0,T]; \HminhalfdivGamma)$ such
  that 
  \begin{equation}
    \label{eq:LocWavfun_Maxwell_H}
    \| \bfH_{\widetilde{\bff}_k} \|_{\bfL^2(B_{a,b})} \to \infty
    \qquad \text{and} \qquad
    \| \bfE_{\widetilde{\bff}_k} \|_{\bfL^2(D_{T})}
    + \| \bfH_{\widetilde{\bff}_k} \|_{\bfL^2(D_{T})} \to 0
    \qquad \text{as } k\to\infty \,,
  \end{equation}
  where $(\bfE_{\widetilde{\bff}_k}, \bfH_{\widetilde{\bff}_k})$ denotes 
  the solution to \eqref{eq:IBVP_Maxwell} for $\bff=\widetilde{\bff}_k$. 
\end{theorem}

\begin{proof}
  The arguments that we use to establish
  Theorem~\ref{thm:LocalizedWavefunctions_Maxwell} will be similar to
  those in the proof of
  Theorem~\ref{thm:LocalizedWavefunctions_acoustic}. 
  Therefore, we discuss the important steps only.
  Also, we limit ourselves to the construction
  of~$\{\bff_k\}_{k\in\N}$. 
  For the discussion on $\{\widetilde{\bff}_k\}_{k\in\N}$ one can replace 
  the operator $\E_{B_{a,b}}$ with $\H_{B_{a,b}}$ in the following
  arguments and adjust Lemma~\ref{lmm:radiating_sources_Maxwell_E}
  appropriately. 

  \emph{Step I:} \;
  We can again assume without loss of generality that 
  $B=B(z,\rho)\Subset \Omega_D$ for some~${\rho>0}$ small such 
  that $\widetilde\Omega:=\Omega\setminus\ol{B\cup D}$ is connected.
  Defining the two operators
  $\E_{B_{a,b}}: \Hcal^1_0([0,T];\HminhalfdivGamma) \to \bfL^2(B_{a,b})$
  and
  $\L_{D_T}: \Hcal^1_0([0,T];\HminhalfdivGamma) \to \bfL^2(D_T)^2$
  as in \eqref{eq:DefLabEabHab_Maxwell}, after replacing $B_{a,b}$
  by~$D_T$ for the second one, we prove by contraposition that
  \begin{equation}
    \label{eq:non-inclusion_Maxwell}
    \ran \E^*_{B_{a,b}} \not\subseteq \ran \L^*_{D_{T}} \,.
  \end{equation}

  To see \eqref{eq:non-inclusion_Maxwell}, we assume that
  $\bfh \in \ran\E^*_{B_{a,b}} \cap \ran\L^*_{D_{T}}$, i.e., 
  \begin{equation*}
    \bfh
    \,=\, - \Scal_0 \bigl[
    \pi_\tau[\widetilde{\bfH}_{(\bfj_1,0)}]\big|_{\Gamma_T} \bigr]
    \,=\, - \Scal_0 \bigl[
    \pi_\tau[\widetilde{\bfH}_{\bfK}]\big|_{\Gamma_T} \bigr] \,,
  \end{equation*}
  where
  $(\widetilde{\bfE}_{(\bfj_1,0)},\,\widetilde{\bfH}_{(\bfj_1,0)})$
  and $(\widetilde{\bfE}_\bfK,\,\widetilde{\bfH}_{\bfK})$
  solve the IBVP \eqref{eq:adjointIBVP_Maxwell} with source terms
  $(\eps(x)\Scal_T[\bfj_1], 0)$
  and~$(\eps(x)\Scal_T[\bfk_1],\mu(x)\Scal_T[\bfk_2])$ for some
  $\bfj_1\in \bfL^2(B_{a,b})$ and
  ${\bfK=(\bfk_1,\bfk_2)\in \bfL^2(D_{T})^2}$, respectively.
  Choosing $\rho>0$ sufficiently small, we may assume 
  that~$b-\delta > \max\{\dist_{\widetilde\Omega}(\widetilde\Omega,\Gamma),a\}$ 
  for some $\delta>0$. 
  We use a unique continuation argument to show  
  \begin{equation}
    \label{eq:same_wave-field_Maxwell}
    \widetilde{\bfE}_{(\bfj_1,0)}
    \,=\, \widetilde{\bfE}_\bfK \qquad\text{and} \qquad
    \widetilde{\bfH}_{(\bfj_1,0)}
    \,=\, \widetilde{\bfH}_\bfK \qquad
    \text{in } \widetilde\Omega_{b-\delta,T} \,.
  \end{equation}
  
  To this end, we denote 
  $\widetilde{\bfE}:=\widetilde{\bfE}_{(\bfj_1,0)}-\widetilde{\bfE}_\bfK$
  and
  $\widetilde{\bfH}:=\widetilde{\bfH}_{(\bfj_1,0)}-\widetilde{\bfH}_\bfK$
  in $\Omega_T$ and observe that
  \begin{equation*}
    \Scal_0 \bigl[
    \pi_\tau[\widetilde{\bfH}] \big|_{\Gamma_T} \bigr]
    \,=\, 0  \,, 
  \end{equation*}
  which implies that $\pi_\tau[\widetilde{\bfH}] = 0$ on $\Gamma_T$.
  Next, we extend $\widetilde{\bfE}$ and $\widetilde{\bfH}$ from
  $\widetilde\Omega_{T}$ to $\widetilde\Omega_{2T}$ by zero, and
  denote these extensions again by $\widetilde{\bfE}$ and
  $\widetilde{\bfH}$, respectively. 
  Since $\widetilde{\bfE}|_{t=T}=\widetilde{\bfH}|_{t=T}=0$
  in~$\widetilde\Omega$, we observe that 
  $(\widetilde{\bfE},\widetilde{\bfH})
  \in C([0,2T];\bfH_0(\curl;\widetilde\Omega)\times\bfH(\curl;\widetilde\Omega))
  \cap C^1([0,2T]; \bfL^2(\widetilde\Omega)^2)$ and 
  \begin{align*}
    \eps(x) \di_t\widetilde{\bfE}
    - \curl\widetilde{\bfH}
    &\,=\, 0
    &&\text{in } \widetilde\Omega_{2T} \,, \\
    \mu(x) \di_t\widetilde{\bfH}
    + \curl\widetilde{\bfE}
    &\,=\, 0
    &&\text{in } \widetilde\Omega_{2T} \,, \\
    \bfnu\times\widetilde{\bfE}\big|_{\Gamma_{2T}}
    \,=\, \bfnu\times\widetilde{\bfH}\big|_{\Gamma_{2T}}
    &\,=\, 0
    &&\text{on } \Gamma_{2T}\,. 
  \end{align*}
  Furthermore, we obtain from the vanishing of 
  $\widetilde{\bfE}|_{t=T}$ and $\widetilde{\bfH}|_{t=T}$ that
  \begin{equation*}
    \div(\eps\widetilde{\bfE}) \,=\, \div(\mu\widetilde{\bfH})
    \,=\, 0 \qquad 
    \text{in } \widetilde\Omega_{2T} \,.
  \end{equation*}  
  Applying the unique continuation principle for the Maxwell
  equations from~\cite[Thm.~4.5]{EINT2002} (see
  also~\cite[Cor.~1.2]{Ell2003} for the global version), we have 
  \begin{equation*}
    \widetilde{\bfE} \,=\, \widetilde{\bfH} \,=\, 0 \qquad
    \text{in } \; \bigl\{ (x,t) \in \widetilde\Omega_{2T}
    \;\big|\; \dist_{\widetilde\Omega}(x, \Gamma) \leq T-|t-T| \} \,.
  \end{equation*}
  Therewith, the choice of $\delta$ implies that
  $\widetilde{\bfE} = \widetilde{\bfH} =  0$ in
  $\widetilde\Omega_{b-\delta,\, T}$,
  proving~\eqref{eq:same_wave-field_Maxwell}. 

  Now we show that
  \begin{equation}
    \label{eq:conseq_intersection_Maxwell}
    \widetilde{\bfE}_{(\bfj_1,0)}
    \,=\, \widetilde{\bfH}_{(\bfj_1,0)}
    \,=\, 0 \qquad
    \text{in } \Omega\setminus\ol{B}\times (b-\delta,T) \,.
  \end{equation}
  To this end we define
  \begin{align*}
    ( \bfE_{\mathrm{com}} , \bfH_{\mathrm{com}} )
    \,:=\,  
    \begin{cases}
      (\widetilde{\bfE}_{\bfK}, \widetilde{\bfH}_{\bfK})
      &\text{in } B\times (b-\delta,T) \,, \\
      (\widetilde{\bfE}_{(\bfj_1,0)}, \widetilde{\bfH}_{(\bfj_1,0)})
      &\text{in } D \times (b-\delta, T) \,, \\
      (\widetilde{\bfE}_{(\bfj_1,0)}, \widetilde{\bfH}_{(\bfj_1,0)})
      \,=\,  (\widetilde{\bfE}_{\bfK}, \widetilde{\bfH}_{\bfK})
      &\text{in } \widetilde\Omega_{b-\delta,\, T} \,.
    \end{cases}
  \end{align*}
  We argue as in the proof of
  Theorem~\ref{thm:LocalizedWavefunctions_acoustic} and appeal to the
  well-posedness of the backward homogeneous problem for Maxwell's
  equations to prove $\bfE_{\mathrm{com}}=\bfH_{\mathrm{com}}=0$ in
  $\Omega_{b-\delta,T}$, which
  implies~\eqref{eq:conseq_intersection_Maxwell}. 
  
  \emph{Step II:} \;
  We use Lemma~\ref{lmm:radiating_sources_Maxwell_E} below with
  $\sigma=b-\delta$ and $\tau=\delta$ to choose 
  $\widetilde{\bfj_1}\in \bfL^2(B_{b-\delta,b})$ such
  that~$\widetilde{\bfE}_{(\widetilde{\bfj_1},0)} \not\equiv 0$ in 
  $(\Omega\setminus{\ol{B}})\times(b-\delta,b)$. 
  Here
  $(\widetilde{\bfE}_{(\bfj_1,0)}, \widetilde{\bfH}_{(\bfj_1,0)})$
  denotes the solution to \eqref{eq:adjointIBVP_Maxwell} with source
  term $(\eps(x)\Scal_T[\bfj_1], 0)$.
  By construction, 
  \begin{equation*}
    \bfZ(x,t)
    \,:=\, - \Scal_0 \bigl[
    \pi_\tau[\widetilde{\bfH}_{\bfj_1}] \big|_{\Gamma_T} \bigr]
    \in \ran\E^*_{B_{a,b}}
  \end{equation*}
  However, $\bfZ\not\in\ran\L^*_{D_{T}}$ in view of Step~I. 
  Hence, we have established \eqref{eq:non-inclusion_Maxwell} and the
  proof is complete. 
\end{proof}

In Step~II of the proof of
Theorem~\ref{thm:LocalizedWavefunctions_Maxwell} we have used the
following auxiliary result. 

\begin{lemma}
  \label{lmm:radiating_sources_Maxwell_E} 
  Consider an open subset $B\subset\Omega$ such that
  $\Omega\setminus\ol{B} \neq \emptyset$ and let
  ${0<\sigma<\sigma+\tau\leq T}$. 
  Then there exists $\bfj_1\in \bfL^2(B_{\sigma,\sigma+\tau})$ such that
  $\widetilde{\bfE}_{(\bfj_1,0)}|_{(\Omega\setminus\ol{B})\times(\sigma,\sigma+\tau)}
  \not\equiv 0$, where
  $(\widetilde{\bfE}_{(\bfj_1,0)},\widetilde{\bfH}_{(\bfj_1,0)})
  \in C([0,T]; \bfH_0(\curl;\Omega)\times\bfH(\curl;\Omega))
  \cap C^1([0,T]; \bfL^2(\Omega)^2)$
  denotes the unique weak solution of the IBVP
  \eqref{eq:adjointIBVP_Maxwell} with source term
  $(\eps(x)\Scal_T[\bfj_1], 0)$.
  Moreover, such a $\bfj_1\in \bfL^2(B_{\sigma,\sigma+\tau})$ can be
  constructed. 
\end{lemma}

\begin{proof}
  Denoting
  $\Mcal^{(\tau)} := \{ x\in\Omega\setminus\ol{B} \;|\; \dist_\Omega(x,\partial B)<\tau \}$,
  we introduce the operator  
  \begin{equation}\label{eq:final_time_operator_Maxwell}
    \T_{\sigma,\tau}:\, \bfL^2(B_{\sigma,\sigma+\tau}) \to \bfL^2(\Mcal^{(\tau)}) \,,
    \qquad \bfj_1 \mapsto \P_{\tau}\bigl[\widetilde{\bfE}_{\bfj_1}|_{t=\sigma}\bigr] \,,
  \end{equation}
  where $(\widetilde{\bfE}_{\bfj_1}, \widetilde{\bfH}_{\bfj_1})$
  solves \eqref{eq:adjointIBVP_Maxwell} for some
  $\bfj_1\in\bfL^2(B_{\sigma,\sigma+\tau})$ and $\bfj_2=0$. 
  Moreover, $\P_\tau$ denotes the orthogonal projection from
  $\bfL^2(\Mcal^{(\tau)})$ to its closed subspace
  $\bfL^2(\div_\eps0;\Mcal^{(\tau)})$, defined by 
  \begin{equation}
    \label{eq:DefL2tildeMtau}
    \bfL^2(\div_\eps0;\Mcal^{(\tau)})
    \,:=\, \bigl\{ \bfg\in\bfL^2(\Mcal^{(\tau)}) \;\big|\;  \div(\eps \bfg) = 0
    \quad \text{in } \Omega \bigr\} \,,
  \end{equation}
  with respect to the inner product
  $\langle\cdot,\cdot,\rangle_{\eps}$
  from~\eqref{eq:DefEpsMuScalarProduct}. 

  Since $\bfj_1(\cdot,t)=0$ for almost all $t\in(\sigma+\tau,T)$, we note that
  $\widetilde{\bfE}_{\bfj_1} = \widetilde{\bfH}_{\bfj_1} = 0$ in
  $\Omega_{\sigma+\tau,T}$.
  To prove the existence of the desired
  $\bfj_1\in\bfL^2(B_{\sigma,\sigma+\tau})$, we prove a stronger
  result implying that $\ran\T_{\sigma,\tau}$ is dense
  in~$\bfL^2(\div_\eps0;\Mcal^{(\tau)})$.
  To see this, we rely on a contrapositive argument. 
  Suppose there exists~${0\neq\bfPhi\in \bfL^2(\div_\eps0;\Mcal^{(\tau)})}$
  such that 
  \begin{equation*}
    \int_{\Mcal^{(\tau)}} \eps(x)
    \P_\tau[\widetilde{\bfE}_{\bfj_1}|_{t=\sigma}](x)
    \cdot \bfPhi(x) \dx
    \,=\, 0 \qquad
    \text{for all } \bfj_1\in \bfL^2(B_{\sigma,\sigma+\tau}) \,.
  \end{equation*} 
  For this $\bfPhi$, we now consider the unique weak solution
  $(\bfE,\bfH) \in
  C([\sigma,\sigma+\tau];\bfL^2(\Omega)\times\bfL^2(\Omega))$ solving
  the IBVP
  \begin{subequations}
    \label{eq:IBVP_Maxwell_help}
    \begin{align}
      \eps(x) \di_t\bfE - \curl\bfH
      &\,=\, 0
      &&\text{in } \Omega_{\sigma,\sigma+\tau} \,, \\
      \mu(x) \di_t\bfH + \curl\bfE
      &\,=\, 0
      &&\text{in } \Omega_{\sigma,\sigma+\tau} \,, \\
      \bfE|_{t=\sigma} \,=\,  \bfPhi \,,\quad
      \bfH|_{t=\sigma}
      &\,=\, 0
      &&\text{in } \Omega \,, \\
      \bfnu\times\bfE|_{(\di\Omega)_{\sigma,\sigma+\tau}}
      &\,=\, 0
      &&\text{on } (\di\Omega)_{\sigma,\sigma+\tau} \,,
    \end{align}
  \end{subequations}
  (see~\cite[Thm.~2 and Cor.~1]{Ant2025}). 
  Next, we consider a smooth approximation
  $\{\bfPhi_k\}_{k\in\N} \in C^\infty_0(\Omega)$ of~$\bfPhi$ with
  $\bfPhi_k\to\bfPhi$ in $\bfL^2(\Omega)$  and denote by
  $(\bfE_k,\bfH_k) \in
  C([\sigma,\sigma+\tau]; \bfH_0(\curl;\Omega)\times\bfH(\curl))
  \cap C^1([\sigma,\sigma+\tau]; \bfL^2(\Omega)\times\bfL^2(\Omega))$
  the associated solution to \eqref{eq:IBVP_Maxwell_help} with
  $\bfPhi_k$ instead of $\bfPhi$ (see~\cite[Thm.~5.3]{KirRie2016}).
  Then, it follows from the continuous dependence of these solutions
  on the data (see~\cite[Thm.~2]{Ant2025}) that
  \begin{equation*}
    \bfE \,=\, \lim_{k\to\infty} \bfE_k \,, \qquad
    \bfH \,=\, \lim_{k\to\infty} \bfH_k \qquad
    \text{in } \bfL^2(\Omega_{\sigma,\sigma+\tau}) \,.
  \end{equation*}
  Now, using orthogonality with respect to
  $\langle\cdot,\cdot\rangle_\eps$ and the fact that
  $(\widetilde{\bfE}_{\bfj_1}, \widetilde{\bfH}_{\bfj_1})$
  vanishes at $t=\sigma+\tau$ and that $\bfH_k$ vanishes at $t=\sigma$
  for all~$k\in\N$, we find 
  \begin{equation*}
    \begin{split}
      0
      &\,=\, \int_{\Mcal^{(\tau)}} \eps(x)
      \P_\tau[\widetilde{\bfE}_{\bfj_1}|_{t=\sigma}](x)
      \cdot \bfPhi(x) \dx 
      \,=\, \int_{\Mcal^{(\tau)}} \eps(x)
      \widetilde{\bfE}_{\bfj_1}(x,\sigma)
      \cdot \bfPhi(x) \dx \\ 
      &\,=\, \lim_{k\to\infty} \int_{\Mcal^{(\tau)}} \eps(x)
      \widetilde{\bfE}_{\bfj_1}(x,\sigma) \cdot\bfPhi_k(x) \dx \\
      &\,=\, \lim_{k\to\infty} \int_\sigma^{\sigma+\tau} \int_{\Omega}
      \bigl( 
      \eps(x) \di_t (\bfE_k\cdot\widetilde{\bfE}_{\bfj_1})(x,t)
      + \mu(x) \di_t(\bfH_k\cdot\widetilde{\bfH}_{\bfj_1})(x,t)
      \bigr) \dx \dt \,.
    \end{split}
  \end{equation*}
  Using \eqref{eq:adjointIBVP_Maxwell} and
  \eqref{eq:IBVP_Maxwell_help}, and integrating by parts with respect
  to~$x$ together with the homogeneous boundary conditions gives
  \begin{equation*}
    \begin{split}
      0
      &\,=\, \lim_{k\to\infty} \int_\sigma^{\sigma+\tau} \int_{\Omega}
      \bigl( \curl\bfH_k \cdot \widetilde{\bfE}_{\bfj_1}
      + \bfE_k \cdot \curl\widetilde{\bfH}_{\bfj_1}
      - \bfH_k \cdot \curl\widetilde{\bfE}_{\bfj_1}
      - \widetilde{\bfH}_{\bfj_1} \cdot \curl\bfE_k \bigr) \dx \dt \\
      &\phantom{\,=\,}
      + \lim_{k\to\infty} \int_\sigma^{\sigma+\tau} \int_{\Omega}
      \bigl(\eps(x)\Scal_T[\bfj_1]\bigr) \cdot \bfE_k \dx \dt \\
      &\,=\, \int_\sigma^{\sigma+\tau} \int_{\Omega}
      \bigl(\eps(x)\Scal_T[\bfj_1]\bigr) \cdot \bfE \dx \dt \,,
    \end{split}
  \end{equation*}
  where we omitted the arguments $(x,t)$ for better readability. 
  Recalling that $\bfj_1\in\bfL^2(B_{\sigma,\sigma+\tau})$ vanishes outside
  $(\sigma,\sigma+\tau)$, integration by parts with respect to~$t$
  yields 
  \begin{equation}
    \label{eq:Orthogonalityj1ScalsigmaE}
    \begin{split}
      0
      &\,=\, \int_\sigma^{\sigma+\tau} \int_{\Omega}
      \bigl(\eps(x)\Scal_{\sigma+\tau}[\bfj_1]\bigr)
      \cdot \di_t \bigl(\Scal_{\sigma}\bfE\bigr) \dx \dt \\
      &\,=\, - \int_\sigma^{\sigma+\tau} \int_{\Omega}
      \bfj_1(x,t)
      \cdot \bigl(\eps(x)\Scal_{\sigma}[\bfE](x,t)\bigr) \dx \dt \,.
    \end{split}
  \end{equation}
  Since \eqref{eq:Orthogonalityj1ScalsigmaE} holds for all
  $\bfj_1\in\bfL^2(B_{\sigma,\sigma+\tau})$, we have shown that
  $\Scal_{\sigma}[\bfE] = 0$ in~$B_{\sigma,\sigma+\tau}$, which 
  further implies that $\bfE=0$ in $B_{\sigma,\sigma+\tau}$.
  Thus also $\bfH=0$ in $B_{\sigma,\sigma+\tau}$ 
  by~\eqref{eq:IBVP_Maxwell_help}.

  Now we consider an extension of $(\bfE, \bfH)$ to
  $\Omega_{\sigma-\tau,\sigma+\tau}$ as follows, 
  \begin{equation*}
    ( \bfE_{\text{ex}}(x,t), \bfH_{\text{ex}}(x,t) )
    \,=\, 
    \begin{cases}
      ( \bfE(x,t), \bfH(x,t) ) \,,
      &x\in\Omega\,,\; t\in[\sigma,\sigma+\tau) \,, \\
      ( \bfE(x,2\sigma-t), -\bfH(x,2\sigma-t)) \,,
      &x\in\Omega\,,\; t\in(\sigma-\tau,\sigma) \,.
    \end{cases}
  \end{equation*}
  Note that this extension satisfies
  $(\bfE_{\text{ex}}, \bfH_{\text{ex}})
  \in C([\sigma-\tau,\sigma+\tau];\bfL^2(\Omega)^2)\cap H^1([\sigma-\tau, \sigma+\tau]; \bfH_0(\curl;\Omega)^*\times \bfH(\curl;\Omega)^*)$
  and constitutes a weak solution of
  \begin{align*}
    \eps(x) \di_t\bfE_{\text{ex}}
    - \curl\bfH_{\text{ex}}
    &\,=\, 0
    &&\text{in } \Omega_{\sigma-\tau,\sigma+\tau} \,, \\
    \mu(x) \di_t\bfH_{\text{ex}}
    + \curl\bfE_{\text{ex}}
    &\,=\, 0
    &&\text{in } \Omega_{\sigma-\tau,\sigma+\tau} \,, \\
    \bfE_{\text{ex}}|_{t=\sigma} \,=\,  \bfPhi \,,\quad
    \bfH_{\text{ex}}|_{t=\sigma}
    &\,=\, 0
    &&\text{in } \Omega \,, \\
    \bfnu\times\bfE_{\text{ex}}|_{(\di\Omega)_{\sigma,\sigma+\tau}}
    &\,=\, 0
    &&\text{on } (\di\Omega)_{\sigma,\sigma+\tau} \,.
  \end{align*}
  We observe that the vanishing of $\bfH_{\text{ex}}|_{t=\sigma}$ in
  $\Omega$ and the fact that $\bfPhi\in\bfL^2(\div_\eps0;\Mcal^{(\tau)})$
  imply 
  \begin{equation*}
    \div(\eps\bfE_{\text{ex}})
    \,=\, \div(\mu\bfH_{\text{ex}})
    \,=\, 0 \qquad
    \text{in } \Omega_{\sigma-\tau,\sigma+\tau} \,,
  \end{equation*}
  which can be seen similar to~\cite[Pro.~1]{Ant2025}.
  Since
  \begin{equation*}
    \bfE_{\text{ex}}
    \,=\, \bfH_{\text{ex}} 
    \,=\, 0 \qquad
    \text{in } B_{\sigma-\tau,\sigma+\tau} \,,
  \end{equation*}
  invoking a unique continuation principle for Maxwell's
  equations~\cite[Thm.~4.5]{EINT2002} (see also~\cite[Thm.~1.1]{Ell2003}), 
  we conclude that
  \begin{equation*}
    \bfE_{\text{ex}}(x,t) \,=\, \bfH_{\text{ex}}(x,t) \,=\, 0 \qquad
    \text{in } \;
    \bigl\{(x,t) \in \Omega_{\sigma-\tau,\sigma+\tau}
    \;\big|\; \dist_{\Omega}(x, \partial B) \leq \tau-|t-\sigma| \bigr\} \,.
  \end{equation*}
  In particular, $\bfPhi = \bfE_{\text{ex}}|_{t=\sigma} = 0$
  in~$\Mcal^{(\tau)}$. 
  This contradicts our choice of $\bfPhi$, proving 
  that~$\ran\T_{\sigma,\tau}$ is indeed dense in
  $\bfL^2(\div_\eps0;\Mcal^{(\tau)})$.
  Therefore, the existence of the desired $\bfj_1 \in
  \bfL^2(B^{(1)}_{a,b})$ for
  Lemma~\ref{lmm:radiating_sources_Maxwell_E} is guaranteed.
  
  For the construction of such a source $\bfj_1$, we can rely on a
  minimization argument based on a Tikhonov functional as done in the
  proof of Lemma~\ref{lmm:radiating_sources_acoustic}.
  Without repeating the steps of the proof, we argue that 
  \begin{equation*}
    \bfj_1
    \,:=\, \bigl( \T_{\sigma,\tau}^*\T_{\sigma,\tau}
    + \beta_0 \I \bigr)^{-1} \T_{\sigma,\tau}^*[\bfg]
  \end{equation*}
  is an example for the desired source in
  Lemma~\ref{lmm:radiating_sources_Maxwell_E} for $0\neq\bfg\in
  \bfL^2(\div_\eps0;\Mcal^{(\tau)})$ and $\beta_0>0$ sufficiently small.
  This concludes the proof of Lemma~\ref{lmm:radiating_sources_Maxwell_E}.
\end{proof}

Next we consider the construction of localized waves as in
Theorem~\ref{thm:LocalizedWavefunctions_Maxwell}.   

\begin{corollary}
  \label{cor:LocalizedWavefunctions_Maxwell}
  A sequence of boundary sources
  $\{\bff_k\}_{k\in\N} \subset \Hcal^1_0([0,T]; \HminhalfdivGamma)$
  as in Theorem~\ref{thm:LocalizedWavefunctions_Maxwell} can be
  explicitly constructed as 
  \begin{equation*}
    \bff_k
    \,=\, \frac{1}{\|\,\sqrt[]{\J_k}[\xi]\|^{3/2}} \J_k[\bfxi]
    \quad \text{with} \quad \J_k
    \,:=\,  \bigl(\L^*_{D_{T}}\L_{D_{T}}
    + k^{-1}\I\bigr)^{-1} \,,
    \qquad k\in\N \,,
  \end{equation*}
  where, taking $\beta>0$ small and
  $0< \tau < \min\{b-a, b-\dist_{\Omega_D}(\Omega_D, \Gamma)\}$, we choose 
  $\sigma=b-\tau$ and
  $0\neq\widetilde\bfg\in \bfL^2(\div_\eps0;\Mcal^{(\tau)})$ to 
  define 
  \begin{equation*}
    \bfxi
    \,:=\, \E^*_{B_{a,b}}
    \Bigl[ 
    \bigl(\T_{\sigma,\tau}^*\T_{\sigma,\tau}
    + \beta \I \bigr)^{-1} \T_{\sigma,\tau}^*[\widetilde\bfg]\Bigr] \,.
  \end{equation*}
  Here the operators $\E_{B_{a,b}}$, $\L_{D_{T}}$, and
  $\T_{\sigma,\tau}$ are defined in \eqref{eq:DefLabEabHab_Maxwell}, 
  and \eqref{eq:final_time_operator_Maxwell}.
\end{corollary}

\begin{proof}
  To construct $\{\bff_k\}_{k\in\N}$, we appeal to 
  Lemma~\ref{lmm:Construction} with the choice 
  $\Acal_1 = \E^*_{B_{a,b}}$, $\Acal_2 = \L^*_{D_{T}}$,
  and $\bfxi = \E^*_{B_{a,b}}\left[\widetilde{\bfj_1}\right]$. 
  For $\sigma=b-\tau$  with $\tau < \min\{b-a,\, \dist_{\Omega_D}(\Omega_D, \Gamma)\}$, 
  we consider here
  \begin{align*}
    \widetilde{\bfj_1} 
    \,:=\, \bigl(\T_{\sigma,\tau}^*\T_{\sigma,\tau}
    + \beta \I \bigr)^{-1} \T_{\sigma,\tau}^*[\widetilde\bfg]
  \end{align*}
  after fixing some $0\neq\widetilde\bfg\in \bfL^2(\div_\eps0;\Mcal^{(\tau)})$ 
  and $\beta >0$ sufficiently small. 
  The operator $\T_{\sigma,\tau}$ is introduced 
  in~\eqref{eq:final_time_operator_Maxwell}. 
  From our choice of $\widetilde{\bfj_1}$ and the arguments of Step~II 
  of Theorem~\ref{thm:LocalizedWavefunctions_Maxwell}, it is clear that 
  $\bfxi\notin \ran\L^*_{D_{T}}$. 
  Similar to the discussion in 
  Corollary~\ref{cor:LocalizedWavefunctions_acoustic}, we note that 
  $\sqrt{\J_k}$ makes sense since $\J_k$ is positive. 
  Denoting $\bfeta_k = \J_k[\bfxi]$, we see~$\langle \bfxi , \bfeta_k\rangle
  = \langle \bfxi , \J_k[\bfxi]\rangle
  = \| \sqrt{\J_k}[\bfxi] \|^2$ for $k\in\N$. 
  With this observation, an application of Lemma~\ref{lmm:Construction} 
  concludes the proof of Corollary~\ref{cor:LocalizedWavefunctions_Maxwell}.
\end{proof}

\subsection{Localization in time}
\label{subsec:Localization_Time_Maxwell}

Similar to Section~\ref{subsec:Localization_Time_Acoustic}, we 
can also construct a sequence of boundary data for which the
associated solutions to~\eqref{eq:IBVP_Maxwell} concentrates on any
given open set $B\subset\Omega$ at a sufficiently large time. 

\begin{theorem}
  \label{thm:localizedSolution__same_base_Maxwell}
  Suppose, in addition to our previous assumptions, 
  that $\eps,\mu\in C^2(\ol{\Omega})$.
  Consider an open subset $B\Subset \Omega$ with Lipschitz boundary such that
  $\Omega_B:=\Omega\setminus\ol{B}$ is connected, and let $[a,b]$ and $[c,d]$
  be two subintervals of $[0,T]$ with $[a,b]\cap[c,d]=\emptyset$
  such that  
  \begin{itemize}
  \item[(I)] $\dist_{\Omega}(\Omega, \Gamma) < |b-d|$ \hspace*{13.6em} if $c<a$\,,
  \item[(II)] $\dist_{\Omega_B}(\Omega_B, \Gamma) < b$ \; and \;
    $\dist_B(B,\di B) < |b-c| / 2$ \qquad if $c>a$\,.
  \end{itemize}
  Then, there exist sequences
  $\{\bff_k\}_{k\in\N},\{\widetilde{\bff}_k\}_{k\in\N}
  \subset \Hcal^1_0([0,T]; \HminhalfdivGamma)$ such that
  \begin{align*}
    \|\bfE_{\bff_k}\|_{\bfL^2(B_{a,b})} \to \infty
    \qquad \text{and} \qquad
    \|\bfE_{\bff_k}\|_{L^2(B_{c,d})}
    + \|\bfH_{\bff_k}\|_{L^2(B_{c,d})} \to 0
    \qquad \text{as } k\to\infty \,, \\
    \|\bfH_{\widetilde{\bff}_k}\|_{\bfL^2(B_{a,b})}\to \infty
    \qquad \text{and} \qquad
    \|\bfE_{\widetilde{\bff}_k}\|_{\bfL^2(B_{c,d})}
    + \|\bfH_{\widetilde{\bff}_k}\|_{\bfL^2(B_{c,d})} \to 0
    \qquad \text{as } k\to\infty \,.
  \end{align*}
  Here, $(\bfE_{\bff_k},\bfH_{\bff_k})$ and
  $(\bfE_{\widetilde{\bff}_k},\bfH_{\widetilde{\bff}_k})$ denote the
  solutions to \eqref{eq:IBVP_Maxwell} for $\bff=\bff_k$ and
  $\bff=\widetilde\bff_k$, respectively. 
\end{theorem}

\begin{proof}
  We divide the details of proof of
  Theorem~\ref{thm:localizedSolution__same_base_Maxwell} in two
  cases.
  Our proof relies on an argument similar to that of
  Theorem~\ref{thm:localizedSolution__same_base_acoustic}.
  Also, we only discuss the existence and construction of $\bff_k$
  and choose to avoid the corresponding discussion of
  $\widetilde{\bff}_k$.
  The latter only requires replacing~$\E^*_{{B}_{a,b}}$ by
  $\H^*_{{B}_{a,b}}$ and adjusting 
  Lemma~\ref{lmm:radiating_boundary_sources_Maxwell} appropriately 
  in the following argument. 

  \emph{Case I: ($c<a$)} \;
  The intervals $[a,b]$ and $[c,d]$ being disjoint, we see
  that~$d<a$. 
  The discussion beyond $t=b$ being immaterial to us, it suffices
  to construct boundary sources
  ${\bff_k\in  \Hcal^1_0([0,b];\HminhalfdivGamma)}$.
  From Lemma~\ref{lmm:radiating_boundary_sources_Maxwell} below 
  with $\tau=b-d$, we obtain
  $\widetilde{\bff}\in \Hcal^1_0([0,b-d];\HminhalfdivGamma)$,
  such that $\bfE_{\widetilde{\bff}}|_{t=b-d} \not\equiv 0$ in $B$. 
  Here, we have utilized our assumption 
  that~${\dist_{\Omega}(\Omega, \Gamma) <|b-d|}$.
  Defining $\bff := \Tcal_d[\widetilde\bff]$ by time-translation as
  in \eqref{eq:time-reversal_and_translation}, and extending this
  function by zero to~$[0,T]$,
  gives~$\bff\in \Hcal^1_0([0,T];\HminhalfdivGamma)$
  such that the solution
  $(\bfE_{\bff},\bfH_{\bff})$ to the
  IBVP~\eqref{eq:IBVP_Maxwell}
  satisfies~$\bfE_{\bff}(x,t) = \bfH_{\bff}(x,t)=0$
  for~$(x,t)\in \Omega_d$ and 
  $\bfE_{\bff}|_{t=b} \not\equiv 0$ in $B$.
  Since~$\bfE_{\bff} \in C([0,b]; \bfL^2(\Omega))$, we
  have
  $\|\bfE_{\bff}\|_{\bfL^2(B_{a,b})} \neq 0$.
  This implies that the sequence of boundary data defined by
  $\bff_k = k \bff$ for~$k\in\N$ validates
  Theorem~\ref{thm:localizedSolution__same_base_Maxwell} for Case I.
  
  \emph{Case II: ($a<c$)} \;
  From our assumption that $[a,b]\cap[c,d]=\emptyset$, we also
  have~$b<c$. 
  In view of Lemma~\ref{lmm:Construction}, it suffices to show
  \begin{equation}
    \label{eq:non-inclusion_same_base_Maxwell}
    \ran \E^*_{{B}_{a,b}} \cap \ran\L^*_{{B}_{c,d}} \,=\, \{0\} 
  \end{equation}
  in order to prove
  Theorem~\ref{thm:localizedSolution__same_base_Maxwell}. 
  As before, we use a contrapositive argument to establish
  \eqref{eq:non-inclusion_same_base_Maxwell}. 

  Let us assume
  $\bfh \in \ran \E^*_{{B}_{a,b}}\cap \ran\L^*_{{B}_{c,d}}$ implying 
  \begin{equation*}
    \bfh
    \,=\, - \Scal_0\bigl[ \pi_\tau[\widetilde\bfH_{(\bfj_1,0)}]
    \big|_{\Gamma_T} \big]
    \,=\, - \Scal_0\bigl[ \pi_\tau[\widetilde\bfH_\bfK]
    \big|_{\Gamma_T} \big] \,,
  \end{equation*}
  where  $(\widetilde{\bfE}_{(\bfj_1,0)},\widetilde{\bfH}_{(\bfj_1,0)})$
  and $(\widetilde{\bfE}_\bfK,\widetilde{\bfH}_{\bfK})$
  solve the IBVP~\eqref{eq:adjointIBVP_Maxwell} with source terms
  $(\eps(x)\Scal_T[\bfj_1], 0)$
  and~$(\eps(x)\Scal_T[\bfk_1],\mu(x)\Scal_T[\bfk_2])$ for some
  $\bfj_1\in \bfL^2(B_{a,b})$ and
  $\bfK=(\bfk_1,\bfk_2)\in\bfL^2(B_{c,d})^2$, respectively.
  As we have seen in Remark~\ref{rem:InvariantDefinitions}, we can
  assume without loss of generality that
  ${\div(\eps\bfk_1)=\div(\mu\bfk_2)=0}$.
  Since~$\bfj_1=0$ and $\bfK=(0,0)$ in $B_{d,T}$, we obtain
  $(\widetilde{\bfE}_{(\bfj_1,0)},\widetilde{\bfH}_{(\bfj_1,0)}) = 
  (\widetilde{\bfE}_\bfK,\widetilde{\bfH}_{\bfK}) = (0, 0)$ in
  $\Omega_{d,T}$. 
  
  We define
  $\bfE:=\widetilde\bfE_{(\bfj_1,0)}-\widetilde\bfE_\bfK$ and
  $\bfH:=\widetilde\bfH_{(\bfj_1,0)}-\widetilde\bfH_\bfK$ in
  $\Omega_d$. 
  Extending~$(\bfE,\bfH)$ from $\Omega_d$ to $\Omega_{2d}$ by zero
  and denoting this extension again by $(\bfE,\bfH)$, we see from  
  Lemma~\ref{lmm:AdjointLabEabHab_Maxwell} that
  $(\bfE,\bfH) \in C([0,2d];\bfH_0(\curl;\Omega)\times\bfH(\curl;\Omega))
  \cap C^1([0,2d]; \bfL^2(\Omega)^2)$. 
  Note that $(\bfE,\bfH)$ satisfies
  \begin{align*}
    \eps(x) \di_t\bfE
    - \curl\bfH
    &\,=\, 0
    &&\text{in } (\Omega_B)_{2d} \,, \\
    \mu(x) \di_t\bfH + \curl\bfE
    &= 0
    &&\text{in } (\Omega_B)_{2d} \,, \\
    \bfE|_{t=d} \,=\, \bfH|_{t=d}
    &\,=\, 0
    &&\text{in } \Omega_B \,, \\
    \nu\times\bfE|_{\Gamma_{2d}}
    \,=\, \nu\times\bfH|_{\Gamma_{2d}}
    &\,=\, 0
    &&\text{on } \Gamma_{2d} \,,    
  \end{align*}
  and thus also
  \begin{equation*}
    \div(\eps(x)\bfE)
    \,=\, \div(\mu(x)\bfH)
    \,=\, 0 \qquad
    \text{in } (\Omega_B)_{2d} \,.
  \end{equation*}
  In view of unique continuation for Maxwell's equations~\cite[Thm.~4.5]{EINT2002}
  (see also~\cite[Thm.~1.1]{Ell2003}), we can
  conclude that 
  \begin{equation*}
    \bfE \,=\, \bfH \,=\, 0 \qquad
    \text{in } \; \bigl\{ (x,t)
    \in (\Omega_B)_{2d} \;\big|\;
    \dist_{\Omega_B}(x,\Gamma)\leq d-|t-d| \bigr\} \,,
  \end{equation*}
  which in consideration of our assumption $\dist_{\Omega_B}(\Omega_B,\Gamma)<b$
  implies that 
  \begin{equation*}
    \bfE(x,t) \,=\, \bfH(x,t)  \,=\, 0 \qquad
    \text{for } (x,t) \in (\Omega_B)_{b,d} \,.
  \end{equation*}
  Since $\bfj_1(\cdot,t)=0$ for almost every $t\in(b,T)$, we also
  have
  $(\widetilde{\bfE}_{(\bfj_1,0)},\widetilde{\bfH}_{(\bfj_1,0)}) = (0,0)$
  in $\Omega_{b,d}$, 
  which implies $\widetilde\bfE_\bfK=\widetilde\bfH_\bfK=0$ in
  $(\Omega_B)_{b,d}$.
  To see that $\bfh=0$ on $\Gamma_T$, it is therefore enough to
  prove~$\pi_\tau[\widetilde\bfH_\bfK] = 0$ on $\Gamma_b$. 
  This is what we show next. 

  Let us start with the observation that
  \begin{align*}
    \eps(x) \di_t\widetilde\bfE_\bfK
    - \curl\widetilde\bfH_\bfK
    &\,=\, - \int_c^T \eps(x) \bfk_1(x,s)  \ds
    &&\text{in } \Omega_{b,c} \,, \\
    \mu(x) \di_t\widetilde\bfH_\bfK
    + \curl\widetilde\bfE_\bfK
    &\,=\, - \int_c^T \mu(x)\bfk_2(x,s) \ds
    &&\text{in } \Omega_{b,c} \,.
  \end{align*}
  In particular the source terms do not depend on time. 
  Accordingly, we define
  \begin{equation*}
    \widetilde\bfE_\bfK'(x,t)
    \,:=\, \di_t\widetilde\bfE_\bfK(x,t)
    \qquad\text{and}\qquad 
    \widetilde\bfH_\bfK' (x,t)
    \,:=\, \di_t\widetilde\bfH_\bfK(x,t) \,,
  \end{equation*}
  and we notice that
  $\div(\eps(x)\widetilde\bfE_\bfK')=\div(\mu(x)\widetilde\bfH_\bfK')=0$
  in~$\Omega_{b,c}$ and $\widetilde\bfE_\bfK'=\widetilde\bfH_\bfK'=0$
  in $(\Omega_B)_{b,c}$. 
  Moreover,~$(\widetilde\bfE_\bfK',\widetilde\bfH_\bfK')\in C([0,T]; \bfL^2(\Omega)^2)$.  
  We would like to invoke the unique continuation result from \cite{EINT2002,Ell2003}, 
  which, however, requires more regularity. 
  For this, one can mollify $\widetilde\bfE_\bfK'$ 
  and~$\widetilde\bfH_\bfK'$ in time. 
  Although this is standard, we briefly discuss this here for completeness. 
  Taking $\gamma \in C^\infty_c(\R)$ with $\supp\gamma\subset [-1,1]$ 
  and $\int_\R \gamma(t) \dt = 1$, let us consider 
  \begin{align*}
    \widetilde\bfE'_{\bfK,n} 
    \,=\, \gamma_n \ast_t \bfE_\bfK' \,, \quad 
    \widetilde\bfH'_{\bfK,n} 
    \,=\, \gamma_n \ast_t \bfH_\bfK'  \qquad 
    \text{where } \; \gamma_n(t) \,=\, n \gamma(nt) \; 
    \text{ for } \; t\in\R \,,\; n\in\N \,,
  \end{align*}
  and $\ast_t$ denotes a convolution in time.
  Taking $n\in\N$ large enough, we may assume that
  $\dist_B(B,\di B)<|c_n-b_n|/2$ where $b_n:=b+\frac{1}{n}$ and 
  $c_n:=c-\frac{1}{n}$. 
  In fact, we note that 
  $\widetilde\bfE'_{\bfK,n}=\widetilde\bfH'_{\bfK,n}=0$ in 
  $\widetilde\Omega_{b,d}$, and 
  $(\widetilde\bfE'_{\bfK,n},\widetilde\bfH'_{\bfK,n}) 
  \in \bfH^2(\Omega_{b_n,c_n})$ satisfies the homogeneous 
  Maxwell equations 
  \begin{align*}
    \eps(x)\di_t\widetilde\bfE'_{\bfK,n} 
    - \curl\,\widetilde\bfH'_{\bfK,n}
    &\,=\, 0 \qquad
      \text{in } \Omega_{b,c} \,, \\
    \mu(x) \di_t\widetilde\bfH'_{\bfK,n}
    + \curl\,\widetilde\bfE'_{\bfK,n}
    &\,=\, 0 \qquad
      \text{in } \Omega_{b,c} \,.
  \end{align*}
  The $\bfH^2$-regularity of $\widetilde\bfE'_{\bfK,n}$ (and 
  also~$\widetilde\bfH'_{\bfK,n}$) is achieved by employing 
  elliptic regularity results since 
  ${\curl\curl\widetilde\bfE'_{\bfK,n}(\cdot,t) \in \bfL^2(\Omega)}$ and 
  $\div\widetilde\bfE'_{\bfK,n}(\cdot,t) \in \bfL^2(\Omega)$
  for all $t\in[b_n,c_n]$. 
  For the latter we use the $C^2$-regularity of $\eps$ along with 
  $\div(\eps(x)\widetilde\bfE'_{\bfK,n}) = 0$ in $\Omega_{b_n,c_n}$. 
  Recalling that~$\widetilde\bfE'_{\bfK,n} = \widetilde\bfH'_{\bfK,n} = 0$ in 
  $(\Omega_B)_{b,d}$, 
  we appeal to the unique continuation
  principle~\cite[Thm.~4.5]{EINT2002} applied 
  to~$(\bfE'_{\bfK,n},\bfH'_{\bfK,n})$ (see also~\cite[Thm.1.1]{Ell2003}) 
  to see that 
  \begin{equation*}
    \bfE'_{\bfK,n} = \bfH'_{\bfK,n} \,=\, 0 \qquad
    \text{in } \; \Bigl\{ (x,t) \in B_{b,c}
    \;\Big|\; \dist_\Omega(x,\di B) < \frac{c-b}{2}- \frac{1}{n}  
    + \Bigl| t - \frac{c+b}{2}\Bigr| \Bigr\} \,.
  \end{equation*}
  Since $(\bfE'_{\bfK,n},\bfH'_{\bfK,n}) \to (\bfE'_\bfK,\bfH'_\bfK)$ 
  in $\bfL^2(\Omega_{b,c})^2$ as $n\to\infty$, we have 
  $\widetilde\bfE_\bfK' = \widetilde\bfH_\bfK' = 0$ a.e.\ 
  in~$B\times\{\frac{b+c}{2}\}$. 
  Combining this with 
  $\widetilde\bfE_\bfK'=\widetilde\bfH_\bfK'=0$
  in $(\Omega_B)_{b,c}$, 
  we therefore have that
  $\widetilde\bfE_\bfK'=\widetilde\bfH_\bfK'=0$
  in~$\Omega_{\frac{b+c}{2}}$ from the well-posedness of IBVP 
  satisfied by $(\widetilde\bfE_\bfK', \widetilde\bfH_\bfK')$
  in~$\Omega_{\frac{b+c}{2}}$.
  Since $\widetilde\bfE_\bfK=\widetilde\bfH_\bfK=0$ on 
  $\Omega_B\times\{\frac{b+c}{2}\}$ as well, we obtain by
  integrating with respect to time that
  $\widetilde\bfE_\bfK=\widetilde\bfH_\bfK=0$
  in~$\Omega_{\frac{b+c}{2}}$.
  This further gives $\pi_\tau[\widetilde\bfH_\bfK] = 0$
  on~$\Gamma_b$, and the proof of
  \eqref{eq:non-inclusion_same_base_Maxwell} is complete.  

  For the construction of $\bff_k$, we appeal to 
  Lemma~\ref{lmm:Construction} with a choice of 
  $0\neq \bfxi \in \ran\E^*_{B_{a,b}}$. 
  In view of the arguments presented above, we note that 
  $\bfxi\notin\ran\L^*_{B_{c,d}}$. 
  In order to choose such a~$\bfxi$, we now refer to 
  Lemma~\ref{lmm:radiating_sources_Maxwell_E} after fixing 
  some $0\neq\widetilde\bfg\in\bfL^2(\div_\eps0;\Mcal^{(\tau)})$. 
  Taking $\sigma=b-\tau$ and 
  $\tau < \min\{b-a,\, \dist_{\Omega_B}(\Omega_B,\Gamma)\}$ in 
  Lemma~\ref{lmm:radiating_sources_Maxwell_E}, we consider 
  $\widetilde{\bfj_1} \in L^2(B_{a,b})$ defined by
  \begin{equation*}
    \widetilde{\bfj_1} 
    \,:=\, \bigl(\T_{\sigma,\tau}^*\T_{\sigma,\tau}
    + \beta_0 \I \bigr)^{-1} \T_{\sigma,\tau}^*[\widetilde\bfg] \,.
  \end{equation*}
  In consideration of such $\widetilde{\bfj_1}$, the preceding 
  arguments also yield that the choice 
  $\bfxi := \L^*_{B_{a,b}}[\widetilde{\bfj_1}]$ satisfies 
  $\bfxi \neq 0$. 
  Now we consider
  \begin{equation*}
    \bff_k
    \,:=\, \frac{1}{\|\,\sqrt[]{\J_k}[\bfxi]\|^{3/2}} \J_k[\bfxi]
    \quad \text{with} \quad \J_k
    \,:=\,  \bigl(\L^*_{B_{c,d}}\L_{B_{c,d}}
    + k^{-1}\I\bigr)^{-1} \,,
    \qquad k\in\N \,,
  \end{equation*}
  Once again, we denote $\bfeta_k := \J_k[\bfxi]$ and then use the 
  positivity of $\J_k$ to make sense of $\sqrt{\J_k}$ while 
  applying Lemma~\ref{lmm:Construction}. 
  Here, we have also used the 
  relation~$\langle \bfxi , \bfeta_k\rangle
  = \langle \bfxi , \J_k[\bfxi]\rangle
  = \| \sqrt{\J_k}[\bfxi] \|^2$ for~$k\in\N$. 
  Therefore, the construction of $\bff_k$ as required in 
  Case~II is complete.
\end{proof}

In Case~I of the proof of
Theorem~\ref{thm:localizedSolution__same_base_Maxwell}, we have used
the following auxiliary result.

\begin{lemma}
  \label{lmm:radiating_boundary_sources_Maxwell}
  For $\tau>\dist_\Omega(\Omega,\Gamma)$ and $B\subseteq \Omega$ open, we can
  construct
  $\bff\in \Hcal^1_0([0,\tau];\HminhalfdivGamma)$
  such that $\bfE_{\bff}|_{t=\tau} \not\equiv 0$ in $B$, where
  $(\bfE_{\bff},\bfH_{\bff})$ denotes the solution 
  to~\eqref{eq:IBVP_Maxwell}.
\end{lemma}

We omit details of the proof of
Lemma~\ref{lmm:radiating_boundary_sources_Maxwell}, as the 
arguments are similar to that of 
Lemma~\ref{lmm:radiating_boundary_sources_acoustic}.
Rather, we briefly discuss the construction aspect of 
$\bff\in \Hcal^1_0([0,\tau];\HminhalfdivGamma)$ as in Lemma~\ref{lmm:radiating_boundary_sources_Maxwell}. 
Denoting $\bfL^2(\div_\eps0;\Omega)$ as defined
in~\eqref{eq:DefL2tildeMtau} with $\Mcal^{(\tau)}$ replaced 
by $\Omega$, we employ unique continuation of Maxwell's 
equations to see that the mapping
\begin{equation*}
  \P_{\tau}:\, \Hcal^1_0([0,\tau];\HminhalfdivGamma) 
  \to \bfL^2(\div_\eps0;\Omega) \,,
  \quad \bff \mapsto \bfE_\bff|_{t=\tau} \,
\end{equation*}
has dense range. 
Here $(\bfE_\bff, \bfH_\bff)$ denotes the solution to 
\eqref{eq:IBVP_Maxwell} for $\tau=T$. 
In order to construct $\bff$ such that $\bfE_f \neq 0$ 
in $B \subseteq \Omega$, we first fix 
$\bfPhi\in\bfL^2(\div_\eps0;\Omega)$ such that 
$\bf\Phi\not\equiv 0$ in $B$. 
Note that this can be always done by taking 
$\bfPhi=\eps^{-1}\curl\bfPsi$ where 
$\bfPsi \in C_c^\infty(B)$ with 
$\curl\bfPsi \not\equiv 0$ in $B$.  
With such choice of $\bfPhi$, we rely on Tikhonov 
regularization as discussed in 
Lemma~\ref{lmm:radiating_sources_acoustic} to conclude that 
\begin{equation*}
  \lim_{\beta\to 0+} \P_\tau[\bff_\beta] 
  \,=\, \bfPhi \,, \qquad 
  \text{where } \bff_\beta 
  \,:=\, \bigl( \P_\tau^*\P_\tau + \beta \I \bigr)^{-1}
  \P_\tau^*[\bfPhi] \,.
\end{equation*}
Therefore, $\bff_\beta$ provides an example of $\bff$ as
desired in Lemma~\ref{lmm:radiating_boundary_sources_Maxwell} 
when $\beta>0$ is sufficiently small. 

\section*{Acknowledgments.}
The research for this paper was supported by 
Deutsche Forschungsgemeinschaft (DFG, German Research Foundation)
-- Project-ID 258734477 -- SFB 1173.

\section*{Appendix. An abstract functional analytic result.}
\renewcommand{\theequation}{A.\arabic{equation}}
\renewcommand{\thetheorem}{A.\arabic{theorem}}
\setcounter{equation}{0}
\setcounter{theorem}{0}
\renewcommand{\thefigure}{A.\arabic{figure}}
\setcounter{figure}{0}

Our construction of localized wave functions crucially relies on a
functional analytic result which characterizes conditions for the
range inclusion of two operators by means of comparing the pointwise
norms of their adjoint operators. 
We recall this for the readers convenience. 

\begin{lemma}[Prop.~12.1.2. of~\cite{TucWei2009}]
  \label{lmm:RangeInclusion}
  For $i=1,2$, consider the bounded linear map $\Acal_i:\, Y_i\to X$
  where $X, Y_1$ and $Y_2$ denote three Hilbert spaces.
  Then we have $\ran\Acal_1 \subs \ran\Acal_2$ if and only if there
  exists a $C>0$ such that  
  \begin{equation*}
    \|\Acal^*_1[\xi]\|_{Y_1} \leq C\|\Acal^*_2[\xi]\|_{Y_2}
    \qquad  \text{for all } \xi\in X \,,
  \end{equation*}
  where $\Acal^*_i$ denotes the adjoint to $\Acal_i$ for $i=1,2$.
\end{lemma}
We further note that Lemma~\ref{lmm:RangeInclusion} is valid even
for reflexive Banach spaces.
We refer the reader to consult~\cite[Lmm.~2.5]{Geb2008} for a proof.

Next we provide a result regarding the construction of localized wave
functions in the abstract framework
of~Lemma~\ref{lmm:RangeInclusion}. 
Such a construction has been already derived in~\cite[Lmm.~2.8]{Geb2008}
under the additional assumption that the adjoint operators in
Lemma~\ref{lmm:Construction}, i.e., $\mathcal{A}_i^*$, $i=1,2,$ are
injective. 

\begin{lemma}
  \label{lmm:Construction}
  Let $\Acal_i$, $i=1,2$, be the operators as defined in
  Lemma~\ref{lmm:RangeInclusion}.
  Suppose $\xi\in\ran\Acal_1$ but $\xi\notin\ran\Acal_2$.
  Then, we have $\lim_{\alpha\to 0+} \|\Acal^*_1[\xi_{\alpha}]\| = \infty$
  and  $\lim_{\alpha\to 0+} \Acal^*_2[\xi_\alpha] = 0$, where  
  \begin{equation*}
    \xi_\alpha
    \,:=\, \frac{\eta_\alpha}{\langle \xi,\eta_\alpha\rangle^{3/4}} \,,
    \quad \eta_\alpha
    \,:=\, \left(\Acal_2\Acal^*_2 + \alpha\I\right)^{-1}[\xi] \,, 
    \quad \alpha>0 \,.
  \end{equation*}
\end{lemma}

\begin{proof}
  Note that, the definition of $\xi_\alpha$ makes sense since
  $\langle \xi,\eta_\alpha\rangle >0$, which follows from the relation 
  \begin{equation}
    \label{eq:ProofConstruction1}
    \langle \xi , \eta_\alpha\rangle
    \,=\, \|\Acal^*_2[\eta_\alpha]\|^2 + \alpha \|\eta_\alpha\|^2 \,,
    \qquad \alpha>0
  \end{equation}
  and the fact that $\xi\neq 0$.
  This implies $\eta_\alpha \neq 0$ making the left hand side of
  \eqref{eq:ProofConstruction1} positive.
  In order to prove Lemma~\ref{lmm:Construction}, we start with the
  observation that 
  \begin{equation}
    \label{eq:ProofConstruction2}
    \lim_{\alpha\to 0+} \langle \xi , \eta_\alpha \rangle
    \,=\, \infty \,.
  \end{equation}
  To see this, we assume the contrary, i.e.,
  $\{\langle \xi , \eta_{\alpha_k} \rangle\}_{k\in\N}$ is bounded for some
  sequence $\alpha_k \to 0+$.
  Making use of \eqref{eq:ProofConstruction1}, we therefore
  see that the sequences $\{\Acal^*_2[\eta_{\alpha_k}]\}_{k\in\N}$ and
  $\{ \sqrt{\alpha_k} \, \eta_{\alpha_k} \}_{k\in\N}$ are bounded in $Y_2$
  and $X$ respectively. 
  An application of the Banach-Alaoglu theorem implies that both these
  sequences admit weakly convergent subsequences which we may
  consider to be same up to a subsequence.
  With an abuse of notation, we still denote this subsequence by
  $\{\alpha_k\}_{k\in\N}$.
  Summarizing, we have
  \begin{equation}
    \label{eq:ProofConstruction3}
    \sqrt{\alpha_k} \, \eta_{\alpha_k} \rightharpoonup \xi_0 \,,
    \qquad \Acal^*_2[\eta_{\alpha_k}] \rightharpoonup \psi_2 \,,
  \end{equation}
  for some $\xi_0\in X$ and $\psi_2 \in Y_2$.
  For any $\widetilde \xi\in X$, we use \eqref{eq:ProofConstruction3}
  to compute 
  \begin{equation*}
    \begin{split}
      \langle \widetilde \xi , \Acal_2[\psi_2]\rangle
      &\,=\, \lim_{k\to\infty} \langle \widetilde \xi ,
      \Acal_2\Acal^*_2[\eta_{\alpha_k}]\rangle
      \,=\, \langle \widetilde \xi , \xi \rangle
      - \lim_{k\to\infty} \sqrt{\alpha_k} \,
      \langle \widetilde \xi , \sqrt{\alpha_k} \eta_{\alpha_k} \rangle \\
      &\,=\, \langle \widetilde \xi , \xi \rangle
      - \langle \widetilde \xi, \xi_0\rangle \lim_{k\to\infty} \sqrt{\alpha_k} 
      \,=\, \langle \widetilde \xi, \xi \rangle \,,
    \end{split}
  \end{equation*}
  implying $\xi=\Acal_2\psi_2$.
  However, this contradicts our assumption that $\xi\notin\ran\Acal_2$
  hence proving~\eqref{eq:ProofConstruction2}. 
  
  Now we calculate 
  \begin{equation*}
    \|\Acal^*_2 [\xi_\alpha]\|
    \,=\, \frac{1}{\langle \xi,\eta_\alpha\rangle^{3/4}}
    \|\Acal^*_2 [\eta_\alpha]\|
    \,\leq\, \frac{1}{\langle \xi , \eta_\alpha\rangle^{3/4}} \,
    \langle \xi , \eta_\alpha \rangle^{1/2}
    \,=\, \langle \xi,\eta_\alpha\rangle^{-1/4} \,,
  \end{equation*}
  which converges to $0$ as $\alpha\to 0+$ due
  to~\eqref{eq:ProofConstruction2}. 
  Here we also use an upper bound for $\|\Acal^*_2[\eta_\alpha]\|$
  from~\eqref{eq:ProofConstruction1}. 
  Now to show
  $\lim_{\alpha\to 0+} \|\Acal^*_1 [\xi_{\alpha}]\| = \infty$, we recall
  that $\xi\in\ran\Acal_1$, i.e., $\xi=\Acal_1[\psi_1]$ for some
  $\psi_1\in Y_1$.
  Employing the Cauchy-Schwarz inequality, we next observe that 
  \begin{equation*}
    \|\Acal^*_1 [\xi_{\alpha}]\|
    \,\ge\, \frac{\langle \Acal^*_1 [\xi_{\alpha}], \psi_1\rangle}{\|\psi_1\|}
    \,=\, \frac{\langle \xi, \xi_{\alpha}\rangle}{\|\psi_1\|}
    \,=\, \frac{\langle \xi,\eta_\alpha\rangle^{1/4}}{\|\psi_1\|} \,,
  \end{equation*}
  which readily gives
  $\lim_{\alpha\to 0+} \|\Acal^*_1 [\xi_{\alpha}]\| = \infty$ in
  consideration of \eqref{eq:ProofConstruction2}. 
 \end{proof}
{\small
  \bibliographystyle{abbrvurl}
  \bibliography{localizedwaves}
}

\end{document}

%% file: macros.tex
\newcommand{\field}[1]{{\mathbb{#1}}}

\newcommand{\N}{\field{N}}
\newcommand{\R}{\field{R}}

\newcommand{\E}{\mathbb{E}}

\renewcommand{\H}{\mathbb{H}}
\newcommand{\I}{\mathbb{I}}
\newcommand{\J}{\mathbb{J}}
\renewcommand{\L}{\mathbb{L}}

\renewcommand{\P}{\mathbb{P}}
\newcommand{\T}{\mathbb{T}}

\newcommand{\one}{\mathbbm{1}}

\newcommand{\Acal}{\mathcal{A}}

\newcommand{\Hcal}{\mathcal{H}}

\newcommand{\Mcal}{\mathcal{M}}

\newcommand{\Ocal}{\mathcal{O}}

\newcommand{\Rcal}{\mathcal{R}}
\newcommand{\Scal}{\mathcal{S}}
\newcommand{\Tcal}{\mathcal{T}}

\newcommand{\bs}{\boldsymbol}

\newcommand{\bff}{{\bs f}}
\newcommand{\bfg}{{\bs g}}
\newcommand{\bfh}{{\bs h}}
\newcommand{\bfj}{{\bs j}}
\newcommand{\bfk}{{\bs k}}

\newcommand{\bfu}{{\bs u}}

\newcommand{\bfE}{{\bs E}}

\newcommand{\bfH}{{\bs H}}

\newcommand{\bfJ}{{\bs J}}
\newcommand{\bfK}{{\bs K}}
\newcommand{\bfL}{{\bs L}}

\newcommand{\bfV}{{\bs V}}

\newcommand{\bfZ}{{\bs Z}}

\newcommand{\bfeta}{{\bs\eta}}

\newcommand{\bfnu}{{\bs\nu}}

\newcommand{\bfxi}{{\bs\xi}}

\newcommand{\bfPhi}{{\bs\Phi}}
\newcommand{\bfPsi}{{\bs\Psi}}

\newcommand{\ol}[1]{\overline{#1}}

\newcommand{\subs}{\subseteq} 
\newcommand{\di}{\partial}

\newcommand{\ds}{\, \dif s} 
\newcommand{\dt}{\, \dif t}
 
\newcommand{\dx}{\, \dif x}

\newcommand{\eps}{\varepsilon}

\newcommand{\ph}{\,\cdot\,}

\newcommand{\Rd}{{\R^3}}

\newcommand{\red}[1]{\textcolor{red}{#1}}

\definecolor{ForestGreen}{RGB}{34,139,34}

\newcommand{\HcurlOmega}{H(\curl;\Omega)}

\newcommand{\HminhalfdivGamma}{{\widetilde{\bfH}^{-1/2}(\sdiv;\Gamma)}}

\newcommand{\HminhalfdivOmega}{\bfH^{-1/2}(\sdiv;\di\Omega)}
\newcommand{\HminhalfcurlOmega}{\bfH^{-1/2}(\sscurl;\di\Omega)}

\DeclareMathAlphabet{\mathbi}{\encodingdefault}{\rmdefault}{\bfdefault}{\itdefault}

\DeclareMathOperator{\dif}{d\!}

\DeclareMathOperator{\dist}{d}
\DeclareMathOperator{\diam}{diam}

\DeclareMathOperator{\supp}{supp}

\DeclareMathOperator{\curl}{{\mathbf{curl}}}
\DeclareMathOperator{\diver}{\mathrm{div}}
\renewcommand{\div}{\diver}

\DeclareMathOperator{\sscurl}{\mathrm{Curl}}
\DeclareMathOperator{\sdiv}{\mathrm{Div}}

\DeclareMathOperator{\ran}{Ran}


%% file: localizedwaves.bib
@article {AlbGri2023,
    AUTHOR = {Albicker, Annalena and Griesmaier, Roland},
     TITLE = {Monotonicity in inverse scattering for {M}axwell's equations},
   JOURNAL = {Inverse Probl. Imaging},
  FJOURNAL = {Inverse Problems and Imaging},
    VOLUME = {17},
      YEAR = {2023},
    NUMBER = {1},
     PAGES = {68--105},
       DOI = {10.3934/ipi.2022032},
       URL = {https://doi.org/10.3934/ipi.2022032},
}

@misc{Ant2025,
      title={Well-Posedness and approximation of weak solutions to
                  time dependent {M}axwell's equations with
                  {$L^2$}-data},  
      author={Harbir Antil},
      year={2025},
      eprint={2510.20752},
      archivePrefix={arXiv},
      primaryClass={math.NA},
      url={https://arxiv.org/abs/2510.20752}, 
}

@article {BKLS2008,
    AUTHOR = {Bingham, Kenrick and Kurylev, Yaroslav and Lassas, Matti and
              Siltanen, Samuli},
     TITLE = {Iterative time-reversal control for inverse problems},
   JOURNAL = {Inverse Probl. Imaging},
  FJOURNAL = {Inverse Problems and Imaging},
    VOLUME = {2},
      YEAR = {2008},
    NUMBER = {1},
     PAGES = {63--81},
       DOI = {10.3934/ipi.2008.2.63},
       URL = {https://doi.org/10.3934/ipi.2008.2.63},
}

@article {DahKirLas2009,
    AUTHOR = {Dahl, Matias F. and Kirpichnikova, Anna and Lassas, Matti},
     TITLE = {Focusing waves in unknown media by modified time reversal
              iteration},
   JOURNAL = {SIAM J. Control Optim.},
  FJOURNAL = {SIAM Journal on Control and Optimization},
    VOLUME = {48},
      YEAR = {2009},
    NUMBER = {2},
     PAGES = {839--858},
       DOI = {10.1137/070705192},
       URL = {https://doi.org/10.1137/070705192},
}

@article {EbePoh2024,
    AUTHOR = {Eberle-Blick, Sarah and Pohjola, Valter},
     TITLE = {The monotonicity method for inclusion detection and the time
              harmonic elastic wave equation},
   JOURNAL = {Inverse Problems},
  FJOURNAL = {Inverse Problems. An International Journal on the Theory and
              Practice of Inverse Problems, Inverse Methods and Computerized
              Inversion of Data},
    VOLUME = {40},
      YEAR = {2024},
    NUMBER = {4},
     PAGES = {045018, 43pp},
      ISSN = {0266-5611,1361-6420},
   MRCLASS = {65M32 (65M60)},
  MRNUMBER = {4714197},
MRREVIEWER = {Jenn-Nan\ Wang},
       DOI = {10.1088/1361-6420/ad2901},
       URL = {https://doi.org/10.1088/1361-6420/ad2901},
}

@article {Ell2003,
    AUTHOR = {Eller, Matthias M.},
     TITLE = {Unique continuation for solutions to {M}axwell's system with
              non-analytic anisotropic coefficients},
   JOURNAL = {J. Math. Anal. Appl.},
  FJOURNAL = {Journal of Mathematical Analysis and Applications},
    VOLUME = {284},
      YEAR = {2003},
    NUMBER = {2},
     PAGES = {698--710},
       DOI = {10.1016/S0022-247X(03)00391-3},
       URL = {https://doi.org/10.1016/S0022-247X(03)00391-3},
}

@incollection {EINT2002,
    AUTHOR = {Eller, M. and Isakov, V. and Nakamura, G. and Tataru, D.},
     TITLE = {Uniqueness and stability in the {C}auchy problem for {M}axwell
              and elasticity systems},
 BOOKTITLE = {Nonlinear partial differential equations and their
              applications. {C}oll\`ege de {F}rance {S}eminar, {V}ol. {XIV}
              ({P}aris, 1997/1998)},
     PAGES = {329--349},
 PUBLISHER = {North-Holland, Amsterdam},
      YEAR = {2002},
       DOI = {10.1016/S0168-2024(02)80016-9},
       URL = {https://doi.org/10.1016/S0168-2024(02)80016-9},
}

@article {Geb2008,
    AUTHOR = {Gebauer, Bastian},
     TITLE = {Localized potentials in electrical impedance tomography},
   JOURNAL = {Inverse Probl. Imaging},
  FJOURNAL = {Inverse Problems and Imaging},
    VOLUME = {2},
      YEAR = {2008},
    NUMBER = {2},
     PAGES = {251--269},
       DOI = {10.3934/ipi.2008.2.251},
       URL = {https://doi.org/10.3934/ipi.2008.2.251},
}

@article {GriHar2018,
    AUTHOR = {Griesmaier, Roland and Harrach, Bastian},
     TITLE = {Monotonicity in inverse medium scattering on unbounded
              domains},
   JOURNAL = {SIAM J. Appl. Math.},
  FJOURNAL = {SIAM Journal on Applied Mathematics},
    VOLUME = {78},
      YEAR = {2018},
    NUMBER = {5},
     PAGES = {2533--2557},
       DOI = {10.1137/18M1171679},
       URL = {https://doi.org/10.1137/18M1171679},
}

@book {Han2017,
    AUTHOR = {Hanke, Martin},
     TITLE = {A taste of inverse problems},
 PUBLISHER = {Society for Industrial and Applied Mathematics (SIAM),
              Philadelphia, PA},
      YEAR = {2017},
     PAGES = {viii+162},
       DOI = {10.1137/1.9781611974942.ch1},
       URL = {https://doi.org/10.1137/1.9781611974942.ch1},
}

@article {Har2009,
    AUTHOR = {Harrach, Bastian},
     TITLE = {On uniqueness in diffuse optical tomography},
   JOURNAL = {Inverse Problems},
  FJOURNAL = {Inverse Problems. An International Journal on the Theory and
              Practice of Inverse Problems, Inverse Methods and Computerized
              Inversion of Data},
    VOLUME = {25},
      YEAR = {2009},
    NUMBER = {5},
     PAGES = {055010, 14},
       DOI = {10.1088/0266-5611/25/5/055010},
       URL = {https://doi.org/10.1088/0266-5611/25/5/055010},
}

@article {HarLinLiu2018,
    AUTHOR = {Harrach, Bastian and Lin, Yi-Hsuan and Liu, Hongyu},
     TITLE = {On localizing and concentrating electromagnetic fields},
   JOURNAL = {SIAM J. Appl. Math.},
  FJOURNAL = {SIAM Journal on Applied Mathematics},
    VOLUME = {78},
      YEAR = {2018},
    NUMBER = {5},
     PAGES = {2558--2574},
       DOI = {10.1137/18M1173605},
       URL = {https://doi.org/10.1137/18M1173605},
}

@article {HarPohSal2019a,
    AUTHOR = {Harrach, Bastian and Pohjola, Valter and Salo, Mikko},
     TITLE = {Monotonicity and local uniqueness for the {H}elmholtz
              equation},
   JOURNAL = {Anal. PDE},
  FJOURNAL = {Analysis \& PDE},
    VOLUME = {12},
      YEAR = {2019},
    NUMBER = {7},
     PAGES = {1741--1771},
       DOI = {10.2140/apde.2019.12.1741},
       URL = {https://doi.org/10.2140/apde.2019.12.1741},
}

@article {HarPohSal2019b,
    AUTHOR = {Harrach, Bastian and Pohjola, Valter and Salo, Mikko},
     TITLE = {Dimension bounds in monotonicity methods for the {H}elmholtz
              equation},
   JOURNAL = {SIAM J. Math. Anal.},
  FJOURNAL = {SIAM Journal on Mathematical Analysis},
    VOLUME = {51},
      YEAR = {2019},
    NUMBER = {4},
     PAGES = {2995--3019},
       DOI = {10.1137/19M1240708},
       URL = {https://doi.org/10.1137/19M1240708},
}

@article {HarUlr2013,
    AUTHOR = {Harrach, Bastian and Ullrich, Marcel},
     TITLE = {Monotonicity-based shape reconstruction in electrical
              impedance tomography},
   JOURNAL = {SIAM J. Math. Anal.},
  FJOURNAL = {SIAM Journal on Mathematical Analysis},
    VOLUME = {45},
      YEAR = {2013},
    NUMBER = {6},
     PAGES = {3382--3403},
       DOI = {10.1137/120886984},
       URL = {https://doi.org/10.1137/120886984},
}

@article {HarUlr2017,
    AUTHOR = {Harrach, Bastian and Ullrich, Marcel},
     TITLE = {Local uniqueness for an inverse boundary value problem with
              partial data},
   JOURNAL = {Proc. Amer. Math. Soc.},
  FJOURNAL = {Proceedings of the American Mathematical Society},
    VOLUME = {145},
      YEAR = {2017},
    NUMBER = {3},
     PAGES = {1087--1095},
       DOI = {10.1090/proc/12991},
       URL = {https://doi.org/10.1090/proc/12991},
}

@misc{HarXia2026,
      title={The monotonicity method for the inverse elastic scattering on unbounded domains}, 
      author={Bastian Harrach and Jianli Xiang},
      year={2026},
      eprint={2602.04453},
      archivePrefix={arXiv},
      primaryClass={math.AP},
      url={https://arxiv.org/abs/2602.04453}, 
}

@article {Hor92,
    AUTHOR = {H\"ormander, L.},
     TITLE = {A uniqueness theorem for second order hyperbolic differential
              equations},
   JOURNAL = {Comm. Partial Differential Equations},
  FJOURNAL = {Communications in Partial Differential Equations},
    VOLUME = {17},
      YEAR = {1992},
    NUMBER = {5-6},
     PAGES = {699--714},
       DOI = {10.1080/03605309208820860},
       URL = {https://doi.org/10.1080/03605309208820860},
}

@book {KatKurLas2001,
    AUTHOR = {Katchalov, Alexander and Kurylev, Yaroslav and Lassas, Matti},
     TITLE = {Inverse boundary spectral problems},
    SERIES = {Chapman \& Hall/CRC Monographs and Surveys in Pure and Applied
              Mathematics},
    VOLUME = {123},
 PUBLISHER = {Chapman \& Hall/CRC, Boca Raton, FL},
      YEAR = {2001},
     PAGES = {xx+290},
       DOI = {10.1201/9781420036220},
       URL = {https://doi.org/10.1201/9781420036220},
}

@article {KKLO2021,
    AUTHOR = {Kirpichnikova, Anna and Korpela, Jussi and Lassas, Matti J.
              and Oksanen, Lauri},
     TITLE = {Construction of artificial point sources for a linear wave
              equation in unknown medium},
   JOURNAL = {SIAM J. Control Optim.},
  FJOURNAL = {SIAM Journal on Control and Optimization},
    VOLUME = {59},
      YEAR = {2021},
    NUMBER = {5},
     PAGES = {3737--3761},
       DOI = {10.1137/20M136904X},
       URL = {https://doi.org/10.1137/20M136904X},
}

@book {KirHet2015,
    AUTHOR = {Kirsch, Andreas and Hettlich, Frank},
     TITLE = {The mathematical theory of time-harmonic {M}axwell's
              equations},
 PUBLISHER = {Springer, Cham},
      YEAR = {2015},
     PAGES = {xiv+337},
       DOI = {10.1007/978-3-319-11086-8},
       URL = {https://doi.org/10.1007/978-3-319-11086-8},
}

@article {KirRie2016,
    AUTHOR = {Kirsch, Andreas and Rieder, Andreas},
     TITLE = {Inverse problems for abstract evolution equations with
              applications in electrodynamics and elasticity},
   JOURNAL = {Inverse Problems},
  FJOURNAL = {Inverse Problems. An International Journal on the Theory and
              Practice of Inverse Problems, Inverse Methods and Computerized
              Inversion of Data},
    VOLUME = {32},
      YEAR = {2016},
    NUMBER = {8},
     PAGES = {085001, 24pp},
       DOI = {10.1088/0266-5611/32/8/085001},
       URL = {https://doi.org/10.1088/0266-5611/32/8/085001},
}

@article {KohVog1984,
    AUTHOR = {Kohn, Robert and Vogelius, Michael},
     TITLE = {Determining conductivity by boundary measurements},
   JOURNAL = {Comm. Pure Appl. Math.},
  FJOURNAL = {Communications on Pure and Applied Mathematics},
    VOLUME = {37},
      YEAR = {1984},
    NUMBER = {3},
     PAGES = {289--298},
       DOI = {10.1002/cpa.3160370302},
       URL = {https://doi.org/10.1002/cpa.3160370302},
}

@article {KohVog1985,
    AUTHOR = {Kohn, R. V. and Vogelius, M.},
     TITLE = {Determining conductivity by boundary measurements. {II}.
              {I}nterior results},
   JOURNAL = {Comm. Pure Appl. Math.},
  FJOURNAL = {Communications on Pure and Applied Mathematics},
    VOLUME = {38},
      YEAR = {1985},
    NUMBER = {5},
     PAGES = {643--667},
       DOI = {10.1002/cpa.3160380513},
       URL = {https://doi.org/10.1002/cpa.3160380513},
}

@article {LasLioTri1986,
    AUTHOR = {Lasiecka, I. and Lions, J.-L. and Triggiani, R.},
     TITLE = {Nonhomogeneous boundary value problems for second order
              hyperbolic operators},
   JOURNAL = {J. Math. Pures Appl. (9)},
  FJOURNAL = {Journal de Math\'{e}matiques Pures et Appliqu\'{e}es.
              Neuvi\`eme S\'{e}rie},
    VOLUME = {65},
      YEAR = {1986},
    NUMBER = {2},
     PAGES = {149--192},
}

@misc{LauLea2023,
      title={Lectures on unique continuation for waves}, 
      author={Laurent , C. and L\'eautaud, M.},
      year={2023},
      eprint={2307.02155},
      archivePrefix={arXiv},
      primaryClass={math.AP},
      url={https://arxiv.org/abs/2307.02155}, 
}

@article {Poh2022,
    AUTHOR = {Pohjola, Valter},
     TITLE = {On quantitative {R}unge approximation for the time harmonic
              {M}axwell equations},
   JOURNAL = {Trans. Amer. Math. Soc.},
  FJOURNAL = {Transactions of the American Mathematical Society},
    VOLUME = {375},
      YEAR = {2022},
    NUMBER = {8},
     PAGES = {5727--5751},
       DOI = {10.1090/tran/8662},
       URL = {https://doi.org/10.1090/tran/8662},
}

@article {Rob91,
    AUTHOR = {Robbiano, L.},
     TITLE = {Th\'eor\`eme d'unicit\'e{} adapt\'e{} au contr\^ole des
              solutions des probl\`emes hyperboliques},
   JOURNAL = {Comm. Partial Differential Equations},
  FJOURNAL = {Communications in Partial Differential Equations},
    VOLUME = {16},
      YEAR = {1991},
    NUMBER = {4-5},
     PAGES = {789--800},
       DOI = {10.1080/03605309108820778},
       URL = {https://doi.org/10.1080/03605309108820778},
}

@article {RueSal2019,
    AUTHOR = {R\"{u}land, Angkana and Salo, Mikko},
     TITLE = {Quantitative {R}unge approximation and inverse problems},
   JOURNAL = {Int. Math. Res. Not. IMRN},
  FJOURNAL = {International Mathematics Research Notices. IMRN},
    VOLUME = {2019},
      YEAR = {2019},
    NUMBER = {20},
     PAGES = {6216--6234},
       DOI = {10.1093/imrn/rnx301},
       URL = {https://doi.org/10.1093/imrn/rnx301},
}

@article {Tam06,
    AUTHOR = {Tamburrino, A.},
     TITLE = {Monotonicity based imaging methods for elliptic and parabolic
              inverse problems},
   JOURNAL = {J. Inverse Ill-Posed Probl.},
  FJOURNAL = {Journal of Inverse and Ill-Posed Problems},
    VOLUME = {14},
      YEAR = {2006},
    NUMBER = {6},
     PAGES = {633--642},
       DOI = {10.1163/156939406778474578},
       URL = {https://doi.org/10.1163/156939406778474578},
}

@article {TamRub02,
    AUTHOR = {Tamburrino, A. and Rubinacci, G.},
     TITLE = {A new non-iterative inversion method for electrical resistance
              tomography},
   JOURNAL = {Inverse Problems},
  FJOURNAL = {Inverse Problems. An International Journal on the Theory and
              Practice of Inverse Problems, Inverse Methods and Computerized
              Inversion of Data},
    VOLUME = {18},
      YEAR = {2002},
    NUMBER = {6},
     PAGES = {1809--1829},
       DOI = {10.1088/0266-5611/18/6/323},
       URL = {https://doi.org/10.1088/0266-5611/18/6/323},
}

@article {Tat95,
    AUTHOR = {Tataru, D.},
     TITLE = {Unique continuation for solutions to {PDE}'s; between
              {H}\"ormander's theorem and {H}olmgren's theorem},
   JOURNAL = {Comm. Partial Differential Equations},
  FJOURNAL = {Communications in Partial Differential Equations},
    VOLUME = {20},
      YEAR = {1995},
    NUMBER = {5-6},
     PAGES = {855--884},
       DOI = {10.1080/03605309508821117},
       URL = {https://doi.org/10.1080/03605309508821117},
}

@article {Tat99,
    AUTHOR = {Tataru, Daniel},
     TITLE = {Unique continuation for operators with partially analytic
              coefficients},
   JOURNAL = {J. Math. Pures Appl. (9)},
  FJOURNAL = {Journal de Math\'{e}matiques Pures et Appliqu\'{e}es.
              Neuvi\`eme S\'{e}rie},
    VOLUME = {78},
      YEAR = {1999},
    NUMBER = {5},
     PAGES = {505--521},
       DOI = {10.1016/S0021-7824(99)00016-1},
       URL = {https://doi.org/10.1016/S0021-7824(99)00016-1},
}

@book {TucWei2009,
    AUTHOR = {Tucsnak, Marius and Weiss, George},
     TITLE = {Observation and control for operator semigroups},
 PUBLISHER = {Birkh\"{a}user Verlag, Basel},
      YEAR = {2009},
     PAGES = {xii+483},
       DOI = {10.1007/978-3-7643-8994-9},
       URL = {https://doi.org/10.1007/978-3-7643-8994-9},
}
